\documentclass[reqno]{amsart}

\usepackage[normalem]{ulem}

\usepackage{amsmath,amssymb,color}
\usepackage{amsfonts, amscd, epsfig, amsmath, amssymb,enumerate}
\usepackage{graphicx}
\usepackage{graphics}
\usepackage{color}
\usepackage{mathrsfs}
\usepackage{todonotes}
\usepackage{soul}

 \usepackage[usenames,dvipsnames]{pstricks}
 \usepackage{pstricks-add}
 \usepackage{epsfig}
 \usepackage{pst-grad} 
 \usepackage{pst-plot} 
 \usepackage[space]{grffile} 
 \usepackage{etoolbox} 
\usepackage[margin=3cm]{geometry}
 \makeatletter 
 \patchcmd\Gread@eps{\@inputcheck#1 }{\@inputcheck"#1"\relax}{}{}
 \makeatother

\newtheorem{theorem}{Theorem}[section]
\newtheorem{lemma}[theorem]{Lemma}
\newtheorem{proposition}[theorem]{Proposition}

\newtheorem{remark}[theorem]{Remark}

\usepackage{tikz}
\usetikzlibrary{backgrounds}
\usetikzlibrary{patterns,fadings}
\usetikzlibrary{arrows,decorations.pathmorphing}
\usetikzlibrary{decorations}
\usetikzlibrary{calc}
\usetikzlibrary{shapes.misc}

\usepackage{setspace}

\definecolor{light-gray}{gray}{0.95}
\usepackage{float}
\usepackage[colorlinks=true,linkcolor=blue,citecolor=magenta]{hyperref}
\def\centerarc[#1](#2)(#3:#4:#5){\draw[#1] ($(#2)+({#5*cos(#3)},{#5*sin(#3)})$) arc (#3:#4:#5);}

\newcommand{\vertiii}[1]{{\left\vert\kern-0.25ex\left\vert\kern-0.25ex\left\vert #1 
    \right\vert\kern-0.25ex\right\vert\kern-0.25ex\right\vert}}

\numberwithin{equation}{section}
\numberwithin{figure}{section}

\newcommand{\mc}[1]{{\mathcal #1}}
\newcommand{\mf}[1]{{\mathfrak #1}}

\newcommand{\bb}[1]{{\mathbb #1}}

\newcommand{\<}{\big\langle}
\renewcommand{\>}{\big\rangle}

\renewcommand{\epsilon}{\varepsilon}

\newcommand{\R}{\mathbb R}
\newcommand{\Z}{\mathbb Z}

\renewcommand{\P}{\mathbb P}

\newcommand{\E}{\mathbb E}

\newcommand{\gen}{\mathcal{L}} 
\newcommand{\gens}{\mathcal{S}} 
\newcommand{\gena}{\mathcal{A}}

\allowdisplaybreaks 

\title[Tagged Particle in Asymmetric Exclusion with Long Jumps]{The motion of the Tagged Particle in asymmetric exclusion process with long jumps}
\author{Linjie Zhao}
\email{linjie\_zhao@hust.edu.cn}
\address{School of Mathematics and Statistics, Huazhong University of Science \& Technology, 430074, Wuhan, China.}
\thanks{\textbf{Acknowledgments.}  Zhao thanks the financial support from  the Fundamental Research Funds for the Central Universities in China.}

\keywords{asymmetric exclusion process; long jumps; tagged particle; central limit theorems}

\begin{document}

\maketitle

\begin{abstract}
We prove law of large numbers and invariance principles for the tagged particle in the asymmetric exclusion process with long jumps when the process starts from its equilibrium measure.
\end{abstract}

\section{Introduction}

It has been a longstanding problem to investigate the long time behavior of  a typical particle, usually called the \emph{tagged particle} or \emph{tracer particle}, in interacting particle systems. For exclusion processes which were introduced by Spitzer \cite{spitzer70} as one of the simplest interacting particle systems,  the motion of the tagged particle has naturally attracted much attention.  
Law of large numbers for the tagged particle on the lattice was proved by Saada \cite{saada1987tagged} and Rezakhanlou \cite{rezakhanlou1994evolution}.   When the process starts from its equilibrium measure,  central limit theorems and invariance principles were proved by  Arratia \cite{arratia1983motion} and Peligrad and Sethuraman \cite{peligrad2007fractional} in the one-dimensional nearest neighbor case, by Kipnis and Varadhan \cite{kipnis1986central} in the symmetric case in all dimensions except the one-dimensional nearest neighbor case, by Varadhan \cite{varadhan1995self} in the asymmetric mean-zero case, by Kipnis  \cite{Kipnis86} in the one-dimensional nearest neighbor asymmetric case  and by Sethuraman, Varadhan and Yau \cite{sethuraman2000diffusive} in the asymmetric case in three or higher dimensions.  In dimensions $d \leq 2$ when the underlying random walk has a drift except the one-dimensional nearest-neighbor case, we only know the tagged particle is diffusive as shown by Sethuraman \cite{sethuraman2006diffusive}, and a full CLT or invariance principle remains open.

We underline that the previous  work all assumed the transition probability of the underlying random walk has finite range, or at least has finite second moments. In this article, we consider the exclusion process with long jumps, whose symmetric version was first introduced by Jara \cite{jara2008longjumps} to derive fractional parabolic equations from interacting particle systems. Recently, hydrodynamic limits for the asymmetric exclusion process with long jumps were proved by Sethuraman  and Shahar \cite{sethuraman2018hydrodynamic}.  By considering its density fluctuations, Gon{\c c}alves and Jara \cite{gonccalves2018density} derived fractional Ornstein–Uhlenbeck process and fractional stochastic Burgers equation from the model.  Roughly speaking, in the long-jump case, a particle jumps from $x$ to $y$ at rate proportional to $\|x-y\|^{-d-\alpha}$, where $d$ is the dimension of the lattice and $\alpha > 0$ is some parameter.   The motion of the tagged particle in the symmetric exclusion process with long jumps has been investigated by Jara \cite{jara2009nonequilibrium}. Hence, we focus on the asymmetric case. 

We start the process from its equilibrium measure, \emph{i.e.}, the Bernoulli product measure with fixed density and conditioned on having a particle at the origin.  We prove law of large numbers (see Theorem \ref{thm:LLN}) and invariance principles (see Theorem \ref{thm:invariant}) for the tagged particle in different regimes of $\alpha$ and $d$. Since we are interested in the long-jump case, we assume $\alpha \leq 2$. We believe the same results should hold as in the finite range case if $\alpha > 2$. We also remark that invariance principles for $d=1, 3/2 \leq \alpha \leq 2$ remain open.

Since the tagged particle itself is not Markovian,  it is more convenient to consider the environment process as seen from the tagged particle, which turns out to be Markovian if the underlying transition probability is translation invariant. The main idea of our proof is to treat the exponential martingales associated with the tagged particle as in the symmetric case \cite{jara2009nonequilibrium}.  However, new difficulties arise due to the irreversibility of the asymmetric exclusion process, and we need to deal with occupation times of the environment process in some regimes of $d$ and $\alpha$.  To control the variance of the occupation times of the environment process, we compare them with the occupation times of the exclusion process with long jumps, which has been investigated previously by Bernardin,  Gon{\c c}alves and Sethuraman \cite{bernardin2016occupation}.

In \cite{jara2009nonequilibrium}, Jara also considered nonequilibrium central limit theorems for the tagged particle in the symmetric exclusion process with long jumps, whose techniques  heavily rely on the hydrodynamic limits for the underlying dynamics. Since the hydrodynamic limits for the asymmetric exclusion with long jumps are not fully proved mainly due to the fact that the uniqueness of the corresponding hydrodynamic equation remains open \cite{sethuraman2018hydrodynamic}, we leave it as a future work to prove nonequilibrium fluctuations for the tagged particle in this setting.

There has been a lot of  work concerning the behavior of the tagged particle in exclusion processes and thus it is impossible for us to list all of them.  At the end of this introduction, we mention the work of Jara and Landim \cite{jara2006nonequilibrium,jara2008quenched}, where nonequilibrium fluctuations were considered for the one-dimensional nearest neighbor exclusion process with/without bond disorder.  Large and moderate deviation principles were also derived respectively in the one-dimensional nearest neighbor case by Sethuraman and Varadhan \cite{sethuraman2013large} and by Xue and the author \cite{xue2022moderate}.

The rest of the article is organized as follows.  In Section \ref{sec:results} we introduce the model and the main results, \emph{i.e.}, law of large numbers  (Theorem \ref{thm:LLN}) and invariance principles (Theorem \ref{thm:invariant}) for the tagged particle. The proof of law of large numbers  is presented in Section \ref{sec:lln}.  In  Section \ref{sec:invariance} we prove convergence in the sense of finite dimensional distributions and in Section \ref{sec:tight} we prove tightness of the tagged particle process, which are sufficient to show the invariance principle for the tagged particle.

\section{Notation and Results}\label{sec:results}

For $d \geq 1$, let $\Z^d$ be the d-dimensional lattice and let $\Z^d_\star = \Z^d \backslash \{0\}$.
For some point $u \in \R^d$, we use $\|u\| =( \sum_{1 \leq j \leq d} u_j^2)^{1/2} $ to denote its $L^2$--norm and $|u| = \max_{1 \leq j \leq d} |u_j|$ to denote its uniform norm, where $u_j$ is the $j$--th coordinate of $u$. Note that the two norms are equivalent, $|u| \leq \|u\| \leq \sqrt{d} |u|$.

The exclusion process is a Markov process defined on the state space $\Omega^d = \{0,1\}^{\Z^d}$, whose infinitesimal generator acting on \emph{local} functions $f: \Omega^d \rightarrow \R$ as
\[	\mathbb{L} f (\eta) = \sum_{x,y \in \Z^d} p(y-x) \eta (x) (1-\eta(y)) [f(\eta^{x,y}) - f(\eta)]\]
where $\eta^{x,y}$ is the configuration obtained from $\eta$ after swapping the values of $\eta(x)$ and $\eta (y)$, 
\[\eta^{x,y} (z) = \begin{cases}
	\eta(x), \quad&z=y,\\
	\eta(y), \quad&z=x,\\
	\eta(z), \quad&z \neq x,y,
\end{cases}\]
and $p(\cdot)$ is some transition probability kernel on $\Z^d$.  Above, $f$ is local if and only if the value of $f$ depends only on finite sites of $\Z^d$. In this article, we are interested in  the asymmetric exclusion with long jumps, where the transition  kernel $p(\cdot)$ is given by
\begin{equation}\label{p}
p(z) = \frac{1}{\|z\|^{d+\alpha}} \Big[2\chi_{\{z_1 > 0\}} + \chi_{\{z_1 = 0\}} \Big], \quad \alpha > 0.
\end{equation}
We set $p(0) = 0$. Above, $\chi_A$ is the indicator function of the set $A$.  Note that $\alpha > 0$ is necessary in order $p(z)$ to be summable. We refer the readers to \cite{liggettips} for construction of the above exclusion process. Denote by $\{\eta_t\}_{t \geq 0}$ the Markov process with generator $\bb{L}$.

\begin{remark}
In general, we only need the transition kernel $p(\cdot)$ to be such that
\begin{enumerate}[(i)]
	\item $p(\cdot)$ could be extended to $\R^d\backslash\{0\}$, whose extension is still denoted by $p(\cdot)$ without confusion, such that $p(u/N) = N^{d+\alpha} p (u)$ for $u\neq 0$ and $N \geq 1$;
	\item the symmetric part of $p(\cdot)$ satisfies $(p(z)+p(-z))/2 = \|z\|^{-d-\alpha}$ for $z \in \Z^d_\star$.
\end{enumerate}
We assume $p(\cdot)$ to take the form \eqref{p} for simplicity.
\end{remark}

For $\rho \in [0,1]$, let $\nu_\rho$ be the  Bernoulli product measure on $\Omega^d$ with constant density $\rho$. It is well known that $\nu_\rho$ is invariant and ergodic for the evolution of the process $\{\eta_t\}_{t \geq 0}$ (see \cite{liggettips} for example).

We shall investigate the behavior of the tagged particle. More precisely, we initially put a particle at the origin and denote by $X_t$ the position of the \emph{tagged} particle at time $t$. Obviously we have $X_0 = 0$. Note that $\{X_t\}_{t \geq 0}$ itself is not a Markov process. Therefore, it is more convenient to work on the \emph{environment process}  as seen from the tagged particle, which  is defined as 
\[\xi_t (z) = \eta_t (X_t + z), \quad z \in \Z^d_\star.\]
Then it is easy to see that $\{\xi_t\}_{t\geq0}$ is a Markov process on the state space  $\Omega^d _\star= \{0,1\}^{\Z^d_\star}$ with  infinitesimal generator acting on local functions  $f: \Omega^d _\star \rightarrow \R$ as 
\[	\gen f (\xi) = \sum_{x,y \in \Z^d_\star} p(y-x) \xi(x) (1-\xi(y)) [f(\xi^{x,y}) - f(\xi)]
+ \sum_{z\in \Z^d_\star}  p(z) (1-\xi (z)) [f(\theta_z \xi) - f(\xi)],\]
where $\xi^{x,y}$ is the environment configuration obtained from $\xi$ after swapping the values of $\xi(x)$ and $\xi(y)$, and $\theta_z \xi$ is the one obtained from $\xi$ after the tagged particle jumps to $z$,  $(\theta_z \xi) (x) = \xi (x+z)$ for $x \neq -z$ and $(\theta_z \xi) (-z) = 0$.
For $\rho \in [0,1]$, denote by $\nu_\rho^\star$ the measure of $\nu_\rho$ conditioned on having a particle at the origin, 
\[\nu_\rho^\star (\cdot) = \nu_\rho (\;\cdot\; | \eta(0)=1).\]
 It is also well known \cite{liggettips} that $\nu_\rho^\star$ is invariant and ergodic for the  environment process $\{\xi_t\}_{t \geq 0}$.
 
Fix $T > 0$. For any probability measure $\mu$ on $\Omega^d$, denote by $\P_\mu$ the probability measure on the path space $D([0,T],\Omega^d)$ induced by the process $\{\eta_t\}_{t \geq 0}$ and the initial measure $\mu$, and by $\E_\mu$ the corresponding expectation. 

The first result of the article is about law of large numbers for the tagged particle when the environment process starts from its equilibrium.

\begin{theorem}[Law of large numbers]\label{thm:LLN}
Suppose the initial measure of the process $\{\eta_t\}_{t \geq 0}$ is $\nu_\rho^\star$ for some $\rho \in (0,1)$. 
\begin{enumerate}[(i)]
		\item if $\alpha = 1$, then 
	\begin{equation}\label{thm:lln_cri}
		\lim_{N \rightarrow \infty} \frac{X_{tN/\log N}}{N} = \gamma_d e_1
	\end{equation}
	in $\P_{\nu_\rho^\star}$--probability, where $e_1 \in \R^d$ is the unit vector with the first component equal to one and the others equal to zero, and 
	\[\gamma_d := \lim_{N \rightarrow \infty}   \frac{1}{\log N} \sum_{\|z\| \leq N} z_1 p(z). \]
	\item  if $\alpha >1$, then
	\begin{equation}\label{thm:lln_sup}
		\lim_{N \rightarrow \infty} \frac{X_{tN}}{N} = t(1-\rho) m 
	\end{equation}
	in $\P_{\nu_\rho^\star}$--probability, where $m := \sum_{z \in \Z^d_\star} z p(z)$.
\end{enumerate}
\end{theorem}

\begin{remark}
	For $0< \alpha < 1$, from Theorem \ref{thm:invariant} below, for any $\varepsilon > 0$,
	\[\lim_{N \rightarrow \infty} \frac{X_{tN^{\alpha-\varepsilon}}}{N} = 0 \]
	in $\P_{\nu_\rho^\star}$--probability.
\end{remark}

The second result  concerns about invariance principles for the tagged particle when the environment process starts from its equilibrium.  For $N \geq 1$, define
\[\overline{X}^N_t = \begin{cases}
X_{tN^\alpha}, \quad &\text{if }0 < \alpha < 1,\\
X_{tN} - tN(1-\rho)  \sum_{\|z\| \leq N} z p(z), \quad &\text{if }  \alpha =1,\\
X_{tN^\alpha} - tN^\alpha(1-\rho)  m , \quad &\text{if } 1 < \alpha < 2,\\
X_{tN^2/\log N} - t(N^2 / \log N)(1-\rho)  m , \quad &\text{if } \alpha = 2.
\end{cases}\]
For $\alpha > 0$, let $\{Y_{\alpha,t}\}_{t \geq 0}$ be the L{\' e}vy process such that  $Y_{\alpha,0} = 0$ and that for any $\beta \in \R$ and $a \in \R^d$,
\[\log E \big[\exp \{i \beta (Y_{\alpha,t} \cdot a) \}\big] = t (1- \rho) \Phi_{\alpha,a} (\beta),\]
where
\begin{enumerate}[(i)]
	\item if $0 < \alpha < 1$, then
	\[\Phi_{\alpha,a} (\beta) = \int_{u_1 > 0} \|u\|^{-d-\alpha} (e^{i \beta( u \cdot a)} - 1)\,du.\]
		\item if $ \alpha = 1$, then
	\[\Phi_{\alpha,a} (\beta) = \int_{u_1 > 0} \|u\|^{-d-1} \big(e^{i \beta( u \cdot a)} - 1- i \beta ( u \cdot a) \chi_{\{\|u\|\leq 1\}} \big)\,du.\]
	\item if $1 < \alpha < 2$, then 
	\[\Phi_{\alpha,a}  (\beta) = \int_{u_1 > 0} \|u\|^{-d-\alpha} \big(e^{i \beta (u \cdot a) } -1 - i \beta (u \cdot a)\big)\,du.\]
	\item if $\alpha =  2$, then 
	\[\Phi_{\alpha,a}  (\beta) =- \frac{\beta^2 }{2} a^T \cdot Da. \]
\end{enumerate}
Above, $D = (D_{i,j})_{1 \leq i,j \leq d}$ is a $d \times d$ matrix given by
\[D_{i,j} = \lim_{N \rightarrow \infty} \frac{1}{\log N} \sum_{\|z\|\leq N} z_i z_j p(z).\]
We underline that in the case $\alpha =  2$, $Y_{\alpha,t}$ is actually a $d$--dimensional Brownian motion with mean zero and covariance matrix $(1-\rho)D$. 

\begin{theorem}[Invariance principle]\label{thm:invariant}
Under the assumptions in Theorem \ref{thm:LLN},  if $d=1$ and $0 < \alpha < 3/2$, or $d \geq 2$ and $0 < \alpha \leq 2$, then 
\[\big\{N^{-1}\overline{X}_{t}^N  \big\}_{0 \leq t \leq T} \Rightarrow \{Y_{\alpha,t}\}_{0 \leq t \leq T}, \quad N \rightarrow \infty.\]
Above, the convergence is in the sense of  Skorohod topology in the path space $D([0,T],\R^d)$.
\end{theorem}

\begin{remark}
	As mentioned in the introduction, the case $\alpha > 2$ should be the same as in the asymmetric finite range case \cite{Kipnis86,sethuraman2000diffusive,sethuraman2006diffusive}. Also note that the case $d=1, 3/2 \leq \alpha \leq 2$ remains unsolved.
\end{remark}

\section{Law of large numbers}\label{sec:lln}

 In this section, we prove law of large numbers for the tagged particle (Theorem \ref{thm:LLN}). Along the proof, the constant $C$ may be different from line to line, but is independent of $N$. For $z \in \Z^d_\star$, let $N^z_t$ be the number of jumps of the tagged particle in the direction $z$ up to time $t$. It  follows immediately that
\begin{equation}\label{X_t}
	X_t = \sum_{z \in \Z^d_\star} z N_t^z.
\end{equation}
Since $N_t^z$ is a compound Poisson process with intensity $\int_0^t p(z) (1-\xi_s (z)) ds$ and   there are no common jumps between the processes $\{N_t^z\}_{z \in \Z^d_\star}$,  for any $a \in \R^d$, any $\beta \in \R$, and any $\gamma = \gamma (N)$, 
\begin{equation}\label{mn}
	\mc{M}_{t}^N (\beta) := \mc{M}_{t,a}^{N} (\beta) := \exp \Big\{i \beta (X_{t \gamma (N)} \cdot a) /N - \sum_{z \in \Z^d_\star} (e^{i\beta (z \cdot a)/N} - 1) p(z) \int_0^{t \gamma (N)} (1-\xi_s(z))\,ds\Big\}
\end{equation}
is a mean one complex martingale. The main idea is to prove convergence of the characteristic function by investigating the above exponential martingale.

\begin{proof}[Proof of Theorem \ref{thm:LLN}] We first prove the simpler case $\alpha > 1$. In order to prove \eqref{thm:lln_sup}, it is enough to show that for any $a \in \R^d$ and any $\beta \in \R$,
	\[\lim_{N \rightarrow \infty} \Big|  \E_{\nu_\rho^\star} \Big[ \exp \{ i \beta (X_{tN} \cdot a) /N\} \Big]  -  \exp \big\{  i \beta t(1-\rho) (m \cdot a)\big\} \Big| = 0. \]
	Choose $\gamma (N) = N$ in \eqref{mn}, then we may bound the above absolute value  by
	\[\E_{\nu_\rho^\star} \Big[ \Big| 1 -   \exp \Big\{  i \beta t(1-\rho) (m \cdot a)-   \sum_{z \in \Z^d_\star} (e^{i\beta (z \cdot a)/N} - 1) p(z) \int_0^{t N} (1-\xi_s(z))\,ds \Big\}  \Big| \Big].\]
	For two real numbers $\beta_j \in \R$, $j=1,2$, denote $\mathfrak{R} (\beta_1 + i \beta_2) = \beta_1$ the real part of the complex number $\beta_1 + i \beta_2$. Since $\alpha > 1$ and $|1-\xi(z)|\leq 1$, 
	\begin{multline*}
	\Big| \mathfrak{R} \Big( \sum_{z \in \Z^d_\star} (e^{i\beta (z \cdot a)/N} - 1) p(z) \int_0^{t N} (1-\xi_s(z))\,ds  \Big) \Big|  \leq \frac{t N }{N^{d+\alpha}}\sum_{z \in \Z^d_\star} |\cos(\beta (z \cdot a) / N)-1| p\big(\tfrac{z}{N}\big) \\
	\leq  C tN^{1-\alpha}  \Big( \int_{1}^\infty r^{-1-\alpha} dr + \int_{1/N}^1 r^{1-\alpha} dr \Big)
	\end{multline*}
for some constant $C = C(\beta,a)$. Note that the last line is bounded by $C t (N^{1-\alpha}+N^{-1})$ if $\alpha \neq 2$ and by $ C t N^{-1} \log N $ if $\alpha = 2$. Using dominated convergence theorem, to conclude the proof, we only need to prove that
	\begin{equation}\label{lln1}
		\lim_{N \rightarrow \infty} \sum_{z \in \Z^d_\star} (e^{i\beta (z \cdot a)/N} - 1) p(z) \int_0^{t N} (1-\xi_s(z))\,ds = i \beta t (1-\rho) (m \cdot a)
	\end{equation}
in $\P_{\nu_\rho^\star}$-probability. To this end, for any $K > 0$,  we bound the difference of the terms on both sides of the last line by
	\begin{multline}\label{lln2} 
\sum_{\|z\|>K} \big| e^{i\beta (z \cdot a)/N} - 1\big| p(z) \int_0^{tN} (1-\xi_s (z))\,ds
		\\+ \sum_{\|z\| \leq K} \big| e^{i\beta (z \cdot a)/N} - 1 -i \beta (z \cdot a) /N\big| p(z) \int_0^{tN} (1-\xi_s (z))\,ds\\
		+ \Big| i \beta  \sum_{\|z\| \leq K} (z \cdot a) p(z) \frac{1}{N}\int_0^{tN} (1-\xi_s (z))\,ds - i \beta t (1-\rho) (m \cdot a )\Big|.
	\end{multline}
	Using the basic inequality
	\begin{equation}\label{basic1}
		\Big| e^{i \theta} - \sum_{k=0}^{n} \frac{(i \theta)^k}{k!} \Big| \leq \frac{|\theta|^{n+1}}{(n+1)!}, \quad \forall \theta \in \R
	\end{equation}
	we bound the first two terms in \eqref{lln2} by 
	\[t \beta \sum_{\|z\| > K} |(z \cdot a)|p(z) + \frac{C t \beta^2 K^2 a^2}{N}.\]
Note that the last line converges to zero as $N \rightarrow \infty, K \rightarrow \infty$. Since the process $\{\xi_t\}_{t \geq 0}$ is ergodic, by ergodic theorem, as $N \rightarrow \infty$, the third term in \eqref{lln2} converges in $\P_{\nu_\rho^\star}$-probability  to
	\[\Big| i \beta t (1-\rho)  \sum_{\|z\| > K} (z \cdot a) p(z)  \Big|.\]
The above term converges to zero as $K \rightarrow \infty$ since the first moment of $p(z)$ is finite.  This proves \eqref{thm:lln_sup}.
	
Now we turn to the case $\alpha = 1$. Similar to the proof of Eq. \eqref{thm:lln_sup}, we only need to prove  
	\begin{equation}\label{lln_cri1}
		\lim_{N \rightarrow \infty}  \int_0^{tN/\log N} \sum_{z \in \Z^d_\star} (e^{i \beta (z \cdot a)/N} - 1) p (z) (1 - \xi_s (z)) ds = i \beta t (1-\rho) \gamma_d a_1 \quad 
	\end{equation}
in $\P_{\nu_\rho^\star}$-probability, where $a_1 = a \cdot e_1$ is the first component of  $a$. We first split  the term on the left side  of \eqref{lln_cri1} as 
	\begin{multline}\label{eq:4.3}
	\int_0^{tN/\log N} \sum_{\|z\| > N } (e^{i \beta (z \cdot a)/N} - 1) p (z) (1 - \xi_s (z)) ds \\
	+ \int_0^{tN/\log N} \sum_{\|z\| \leq N  } (e^{i \beta (z \cdot a)/N} - 1 - i \beta (z \cdot a)/N) p (z) (1 - \xi_s (z)) ds \\
	+ \frac{1}{N} \int_0^{tN/\log N} \sum_{\|z\| \leq N  } i \beta (z \cdot a) p (z) (1 - \xi_s (z)) ds.
	\end{multline}
Since $|e^{i \beta (z \cdot a)/N} - 1| \leq 2$, the absolute value of the first term in \eqref{eq:4.3} is bounded by
	\[\frac{Ct}{\log N} \int_{1}^\infty \frac{1}{r^{d+1}} r^{d-1} dr \leq \frac{Ct}{\log N}.\]
Using the inequality \eqref{basic1}, we bound the absolute value of the second term in \eqref{eq:4.3} by
	\[\frac{Ct}{\log N} \int_{\tfrac{1}{N}}^{1} \frac{r^2}{r^{d+1}} r^{d-1} dr \leq \frac{Ct}{\log N}.\]
To deal with the last term in \eqref{eq:4.3}, first note that
	\begin{multline*}
		\lim_{N \rightarrow \infty} \E_{\nu_\rho^\star} \Big[  \frac{1}{N} \int_0^{tN/\log N} \sum_{\|z\| \leq N} i \beta (z \cdot a) p (z) (1 - \xi_s (z)) ds \Big] \\= i \beta t (1-\rho) a \cdot \bigg\{ \lim_{N \rightarrow \infty}   \frac{1}{\log N} \sum_{\|z\| \leq N } z p(z)  \bigg\} =  i \beta t (1-\rho) \gamma_d a_1.
	\end{multline*}
Above, the first identity comes from the invariance of the measure $\nu_\rho^\star$, and the last one comes from the definition of $\gamma_d$. To  conclude the proof, we only need to show
\[\lim_{N \rightarrow \infty} {\rm Var}_{\nu_\rho^\star} \Big(  \frac{1}{N} \int_0^{tN/\log N} \sum_{\|z\| \leq N} i \beta (z \cdot a) p (z) (1 - \xi_s (z)) ds \Big)=0.\]
	By Cauchy-Schwarz inequality, the variance in the last line  is bounded by
	\[\frac{C t^2 \beta^2}{(\log N)^2} \sum_{\|z\| \leq N} |z \cdot a|^2 p(z)^2 \leq \frac{C t^2 \beta^2 }{N^d (\log N)^2} \int_{\frac{1}{N}}^{1} \frac{r^2}{r^{2d+2}} r^{d-1} dr \leq \frac{C t^2 \beta^2}{ (\log N)^2},\]
which converges to zero as $N \rightarrow \infty$. This  concludes the proof of law of large numbers.
\end{proof}

\section{Invariance Principle}\label{sec:invariance}

In this and the next sections, we prove Theorem \ref{thm:invariant}.  By \cite[Theorem 13.1]{billingsley2013convergence}, we only need to prove the convergence in the sense of finite dimensional distributions and tightness of the tagged particle process. We prove the former in this section and the latter in the next one.  The main aim of this section is to prove the following result.

\begin{proposition}\label{pro:charac} Under the assumptions in Theorem \ref{thm:invariant}, for any $a \in \R^d$, any fixed $t \geq 0$, and any $\beta \in \R$,
		\begin{equation}\label{X}
		\lim_{N \rightarrow \infty} \log \E_{\nu_\rho^\star} \Big[ \exp \{ i \beta (\overline{X}_{t}^N \cdot a) /N\} \Big] =   t(1-\rho) \Phi_{\alpha,a} (\beta).
	\end{equation}
\end{proposition}

\begin{remark}
	By the above proposition, the convergence in  Theorem \ref{thm:invariant} holds for any fixed time $t$. The convergence in the sense of finite dimensional distributions follows from the observations in \cite[Remark 3.3]{jara2009nonequilibrium}. Indeed, by exploiting the  martingale $\mathcal{M}^N_{\cdot+s} (\beta) / \mathcal{M}^N_{s} (\beta)$ for any $s \geq 0$ and by following the  proof of the above proposition  line by line, one could prove that the increments of the tagged particle process are conditionally independent and identically distributed, which is sufficient to prove the convergence in the sense of finite dimensional distributions.
\end{remark}

In the rest of this section, we prove Proposition \ref{pro:charac} in different regimes of parameters $\alpha$ and $d$.

\subsection{The case $0 < \alpha < 1$.}  In this case, the jump rate $p(\cdot)$ has heavy tails and the proof is very similar to the symmetric case \cite{jara2009nonequilibrium}. Recall $\overline{X}^N_t = X_{tN^\alpha}$ if $0 < \alpha < 1$. Take $\gamma (N) = N^\alpha$ in \eqref{mn}, then we have
	\[\mc{M}_{t}^N (\beta) = \exp\Big\{ i \beta (\overline{X}^N_t \cdot a)/N - \frac{1}{N^\alpha} \int_0^{tN^\alpha} \frac{1}{N^d} \sum_{z \in \Z^d_\star} p \big(\tfrac{z}{N}\big)(e^{i\beta (z \cdot a) /N} - 1) (1-\xi_s (z)) ds \Big\}.\]
Therefore,
\begin{multline*}
	\Big|  \E_{\nu_\rho^\star} \Big[ \exp \{ i \beta (\overline{X}^N_t \cdot a) /N\} \Big]  -  \exp \big\{  t(1-\rho) \Phi_{\alpha,a} (\beta) \big\} \Big| 
	\\ \leq \E_{\nu_\rho^\star} \Big[ \Big| 1 -   \exp \Big\{  t(1-\rho) \Phi_{\alpha,a} (\beta)-   \frac{1}{N^\alpha} \int_0^{tN^\alpha} \frac{1}{N^d}  \sum_{z \in \Z^d_\star} p \big(\tfrac{z}{N}\big) (e^{i\beta (z \cdot a) /N} - 1) (1-\xi_s (z)) ds \Big\}  \Big| \Big].
\end{multline*}
Since there exists a finite constant $C$ independent of $N$ such that
\begin{equation}\label{realPart}
\bigg|\mathfrak{R} \Big( \frac{1}{N^d}  \sum_{z \in \Z^d_\star} p \big(\tfrac{z}{N}\big)(e^{i\beta (z \cdot a) /N} - 1) (1-\xi_s (z)) \Big) \bigg| \\
\leq C \Big(1+\int_{\tfrac{1}{N}}^1 r^{1-\alpha} dr \Big) \leq C,
\end{equation}
by  dominated convergence theorem,  it suffices to show that 
\begin{equation}\label{vn_esub}
	\lim_{N \rightarrow \infty} \frac{1}{N^\alpha} \int_0^{tN^\alpha}  \frac{1}{N^d}  \sum_{z \in \Z^d_\star} p \big(\tfrac{z}{N}\big) (e^{i\beta (z \cdot a) /N} - 1) (1-\xi_s (z)) ds 
	= t(1-\rho) \Phi_{\alpha,a} (\beta)
\end{equation}
in $\P_{\nu_\rho^\star}$--probability.  Observe that 
\[\lim_{N \rightarrow \infty} \E_{\nu_\rho^\star} \Big[ \frac{1}{N^\alpha} \int_0^{tN^\alpha}  \frac{1}{N^d}  \sum_{z \in \Z^d_\star} p \big(\tfrac{z}{N}\big) (e^{i\beta (z \cdot a) /N} - 1) (1-\xi_s (z)) ds \Big]
= t(1-\rho) \Phi_{\alpha,a} (\beta),\]
and by Cauchy-Schwarz inequality, that
\begin{multline*}
{\rm Var}_{\nu_\rho^\star} \Big( \frac{1}{N^\alpha} \int_0^{tN^\alpha}  \frac{1}{N^d}  \sum_{z \in \Z^d_\star} p \big(\tfrac{z}{N}\big) (e^{i\beta (z \cdot a) /N} - 1) (1-\xi_s (z)) ds \Big) \\ 
\leq  \frac{Ct^2}{N^d} \Big( \int_{\tfrac{1}{N}}^1 \frac{r^2}{r^{2d+2\alpha}} r^{d-1} dr + 1 \Big) \leq Ct^2 \big( N^{2\alpha - 2} + N^{-d}\big),
\end{multline*}
which converges to zero as $N \rightarrow \infty$ since $\alpha < 1$. This proves \eqref{vn_esub} and thus concludes the proof for the case $0 < \alpha < 1$.

\subsection{The case $d \geq 2$ and $1 < \alpha < 2$.}  
Recall $\overline{X}_{t}^N =X_{tN^\alpha} - tN^\alpha (1-\rho) m$ in this case. Take $\gamma (N) = N^\alpha$ in \eqref{mn}, then we may rewrite the martingale $\mathcal{M}^N_t$ defined in \eqref{mn} as 
	\[\mc{M}_{t}^N (\beta) = \exp\Big\{i \beta (\overline{X}_{t}^N \cdot a)  /N -\frac{1}{N^\alpha} \int_0^{tN^\alpha} \Gamma^{N}_a (\xi_s)\,ds \Big\},\]
	where for $a \in \R^d$ and $\xi \in \Omega_\star^d$,
	\begin{equation}\label{GammaN}
		\Gamma^N_a (\xi)= \frac{1}{N^d} \sum_{z \in \Z_\star^d} p \big(\tfrac{z}{N}\big)(e^{i\beta (z \cdot a) /N} - 1) (1-\xi (z)) - i \beta  N^{\alpha -1} (1-\rho) (m \cdot a).
	\end{equation}
Therefore,
	\begin{multline*}
		\Big| \E_{\nu_\rho^\star} \Big[ \exp \Big\{i \beta (\overline{X}_{t}^N \cdot a)  /N\Big\} \Big]  - \exp \{t(1-\rho) \Phi_{\alpha,a} (\beta)\} \Big| \\
		\leq \E_{\nu_\rho^\star}  \Big[  \Big| 1 -  \exp\Big \{t(1-\rho) \Phi_{\alpha,a} (\beta) - \frac{1}{N^\alpha} \int_0^{tN^\alpha} \Gamma^{N}_a (\xi_s)\,ds  \Big\} \Big| \Big].
	\end{multline*}
As in \eqref{realPart}, if $1 < \alpha < 2$, there exists a constant $C$ independent of $N$ such that
\[\Big|\mathfrak{R} \big( 	\Gamma^N_a (\xi) \big) \Big| \\
\leq C \Big(1+\int_{\tfrac{1}{N}}^1 r^{1-\alpha} dr \Big) \leq C.\]
In Lemma  \ref{lem:vN} below, we shall prove that
\[ \lim_{N \rightarrow \infty}  \frac{1}{N^\alpha} \int_0^{tN^\alpha} \Gamma^N_a (\xi_s)\,ds = t (1- \rho) \Phi_{\alpha,a} (\beta)\]
in $\P_{\nu_\rho^\star}$--probability. Then we finish the proof of Proposition \ref{pro:charac} in the case $d \geq 2$ and $1 < \alpha < 2$ by dominated convergence theorem. 

The rest of this subsection is to prove the following crucial lemma.

\begin{lemma}\label{lem:vN}  
 Suppose $d \geq 2$ and $1 < \alpha < 2$. Then, for any $a \in \R^d$, any $t > 0$, and any $\beta \in \R$, 
 \begin{equation}\label{vN_above1}
 \lim_{N \rightarrow \infty}  \frac{1}{N^\alpha} \int_0^{tN^\alpha} \Gamma^N_a (\xi_s)\,ds = t (1- \rho) \Phi_{\alpha,a} (\beta)
 \end{equation}
in $\P_{\nu_\rho^\star}$--probability, where  $\Gamma^N_a$  is defined in \eqref{GammaN}.
\end{lemma}

Before proving the above result, we first state a lemma concerning about the recurrence or transience of random walks with long jumps.

\begin{lemma}\label{lem:rw}
	Let $\alpha > 0$.  Let $\{Z_n\}_{n \geq 1}$ be the symmetric random walk on $\Z^d$ with transition probability 
	\[s (x) = \frac{1}{s^\star} \frac{1}{\|x\|^{d+\alpha}},\quad x \neq 0,\]
	where $s^\star$ is the normalizing constant
	\[s^\star = \sum_{x \in \Z^d,\atop x \neq 0} \frac{1}{\|x\|^{d+\alpha}}.\]
	Then, $\{Z_n\}$ is recurrent if and only if  $d \leq 2$ and  $\alpha \geq d$. 
\end{lemma}

The above results were proved very recently in \cite{baumler2022recurrence} by using the equivalence between the transience of random walks and the existence of a unit flow with finite energy from the origin to infinity. Below we present a more elementary proof by analyzing the characteristic function of the random walk $\{Z_n\}_{n \geq 1}$. 

\begin{proof}
	It is well known that the random walk is transient if $d \geq 3$.  Thus it remains to consider the case $d \leq 2$. It is well known that  the random walk is recurrent if and only if for some fixed $\delta > 0$ (see \cite{durrett2017probability} for example),
	\begin{equation}\label{recurrence}
		\int_{(-\delta,\delta)^d} \mathfrak{R} \Big( \frac{1}{1 - \varphi (\theta)}  \Big) d \theta = \infty,
	\end{equation}
	where $\varphi (\theta)$, $\theta \in \R^d$, is the characteristic function of the random walk $\{Z_n\}$,
	\[\varphi(\theta) = \sum_{x \in \Z^d} e^{i (\theta \cdot x)} s (x) = \sum_{x \in \Z^d} \cos (\theta \cdot x) s (x).\]

We first consider  the simpler case $d=1$. If $\alpha > 1$, then the mean of the transition kernel $s(x)$ is zero and by  law of large numbers,
\[\lim_{n \rightarrow \infty} \frac{1}{n} \sum_{j=1}^n Z_j = 0\]
in probability. By Chung-Fuchs theorem (see \cite{durrett2017probability} for example), the random walk is recurrent in this case. If  $0 < \alpha \leq 1$, observe that
	\[
	\lim_{\theta \rightarrow 0} \frac{1 - \varphi (\theta)}{|\theta|^\alpha} = 	\lim_{\theta \rightarrow 0} \frac{|\theta| }{s^\star}\sum_{x \in \Z} \frac{1 - \cos (\theta  x)}{|\theta x|^{1+\alpha}}
	= \frac{1}{s^\star} \int_\R \frac{1-\cos u}{|u|^{1+\alpha}} du.
	\]
Since the term on the right side in the last line is finite and  positive, there exists a constant $C > 1$ such that
\[	C^{-1}	\int_{(-\delta,\delta)} |\theta|^{-\alpha}  d \theta \leq \int_{(-\delta,\delta)}  \frac{1}{1 - \varphi (\theta)}  d \theta \leq C	\int_{(-\delta,\delta)} |\theta|^{-\alpha}  d \theta.\]
Since the integral  $\int_{(-\delta,\delta)} |\theta|^{-\alpha}  d \theta$ is finite if and only if $0 < \alpha < 1$, the random walk is transient if and only if $0 < \alpha < 1$ in the case $d=1$.
	
We now consider  the case $d=2$, whose proof is similar to the one dimensional case but is more involved. If $\alpha > 2$, then the transition kernel $s(x)$ has finite second moment and thus the central limit theorem holds for $n^{-1/2} \sum_{j=1}^n Z_j$. 
By \cite{durrett2017probability}, the random walk is recurrent in this case.  Now, assume $0 < \alpha < 2$. Fix $v \in \R^2$ such that $\|v\| = 1$.  Define $f_v: \R^2 \rightarrow \R$ as
\[f_v (u) = \frac{1}{s^\star} \frac{1-\cos (v \cdot u)}{\|u\|^{2+\alpha}}, \quad u \in \R^2.\]
Observe that $ \int_{\R^2} f_v (u) du $ is integrable and is independent of $v$. We claim that
\begin{equation}\label{rw0}
\lim_{r \rightarrow 0} \sup_{\|v\|=1} \Big|\frac{1 - \varphi (r v)}{r^\alpha} - \int_{\R^2} f_v (u) du \Big| =0.
\end{equation}

 Indeed, for any $0< r <1$, we have
\begin{multline}\label{rw1}
	\Big|\frac{1 - \varphi (r v)}{r^\alpha} - \int_{\R^2} f_v (u) du \Big| = 	\Big|r^2 \sum_{x \in \Z^2_\star} f_v(rx) - \int_{\R^2} f_v (u) du \Big| \\
	\leq  \sum_{x \in \Z^2_\star} \int_{|u-rx|<r/2} \big| f_v(rx) -  f_v (u) \big| du   + \int_{|u| < r/2} f_v (u) du\\
	\leq  C r  \sum_{x \in \Z^2_\star} \sum_{j=1}^2 \int_{|u-rx|<r/2} \big| \partial_{u_j} f_v(w_{rx,u}) \big| du + C r^{2-\alpha},
\end{multline} 
where $w_{rx,u}$ is some point between $rx$ and $u$, and the constant $C$ above is independent of $v$ and $r$.  Direct calculations yield that for $j=1,2$,
\[\partial_{u_j} f_v (u)= \frac{1}{s^\star} \frac{v_j \sin (v \cdot u)}{\|u\|^{2+\alpha}} - \frac{2+\alpha}{s^\star} \frac{u_j (1-\cos (v \cdot u))}{\|u\|^{4+\alpha}}.\]
Thus, there exists some constant $C$ independent of $v$ such that
\[\big| \partial_{u_j} f_v (u) \big| \leq C \big[ \min \{\|u\|^{-2-\alpha}, \|u\|^{-1-\alpha}\} +  \min \{\|u\|^{-3-\alpha}, \|u\|^{-1-\alpha}\} \big]. \]
Then, we may bound the first term on the right side of \eqref{rw1} by
\begin{equation*}
C r  \Big\{ \sum_{\|rx\|>1}  \int_{|u-rx|<r/2} \big(\|w_{rx,u}\|^{-2-\alpha} + \|w_{rx,u}\|^{-3-\alpha}\big) du
+  \sum_{r \leq \|rx\|\leq 1}  \int_{|u-rx|<r/2} \|w_{rx,u}\|^{-1-\alpha}  du \Big\}.
\end{equation*}
Since $\|w_{rx,u}\|  \geq |w_{rx,u}| \geq |rx| - r/2 \geq |rx|/2 \geq \|rx\|/ (2 \sqrt{2})$, the last line is bounded by
\begin{multline*}
Cr^3  \Big\{  \sum_{\|rx\|>1}  \big(\|rx\|^{-2-\alpha} + \|rx\|^{-3-\alpha}\big) 
+  \sum_{r \leq \|rx\|\leq 1}   \|rx\|^{-1-\alpha}  \Big\} \\
\leq C r \Big\{ \int_1^\infty (\ell^{-1-\alpha} + \ell^{-2-\alpha}) d \ell + \int_r^1 \ell^{-\alpha} d \ell \Big\}.
\end{multline*}
Note that the term on the right side in the last inequality is bounded by $Cr$ if $0 < \alpha < 1$, by $C (r- r \log r)$ if $\alpha = 1$ and by $C (r+r^{2-\alpha})$ if $1 < \alpha < 2$. Since the constant $C$ above is independent of $v$, we conclude the proof of \eqref{rw0}.

By \eqref{rw0}, there exists a constant $C > 1$ such that for any $\theta \in \R^2$ with norm small enough, 
\[C^{-1} \|\theta\|^{-\alpha} \leq 1 -\varphi (\theta) \leq C \|\theta\|^{-\alpha}.\]
Then by \eqref{recurrence}, the random walk is transient if $0 < \alpha < 2$. 

Similarly, for $\alpha = 2$,  one could show that
\[\lim_{r \rightarrow 0} \sup_{\|v\|=1} \Big|\frac{1 - \varphi (r v)}{r^2 |\log r|  } - \frac{1}{|\log r|}\int_{\|u\|>r} f_v (u) du \Big| =0.\]
Therefore,  there exists a constant $C > 1$ such that for any $\theta \in \R^2$ with norm small enough,
\[- C^{-1} \|\theta\|^{-2}  \log \|\theta\| \leq 1 -\varphi (\theta) \leq -C \|\theta\|^{-2}  \log \|\theta\|.\]
Then by \eqref{recurrence}, it is easy to see  the random walk is recurrent if $\alpha = 2$. This concludes the proof.
\end{proof}

\begin{proof}[Proof of Lemma \ref{lem:vN}] Since $p(\tfrac{z}{N}) = N^{d+\alpha} p(z)$ for any $z \in \Z^d_\star$, we could rewrite the term on the left side of \eqref{vN_above1} as
\begin{align}
\frac{1}{N^\alpha} \int_0^{tN^\alpha} \frac{1}{N^d} \sum_{z \in \Z^d\star} p \big(\tfrac{z}{N}\big) (e^{i \beta (z \cdot a) /N} &- 1 - i \beta (z \cdot a) /N) (1 - \xi_s (z)) ds\label{eqn:1}\\
&+ i \beta \sum_{z \in \Z^d_\star} (z \cdot a)  p(z) \frac{1}{N} \int_0^{tN^\alpha} (\rho - \xi_s (z))\,ds.\label{eqn:2}
\end{align}
For \eqref{eqn:1}, it is easy to see that 
\[\lim_{N \rightarrow \infty} \E_{\nu_\rho^\star} \Big[ \frac{1}{N^\alpha} \int_0^{tN^\alpha} \frac{1}{N^d} \sum_{z \in \Z^d_\star} p \big(\tfrac{z}{N}\big) (e^{i \beta (z \cdot a) /N} - 1 - i \beta (z \cdot a) /N) (1 - \xi_s (z)) ds \Big] = t (1-\rho) \Phi_{\alpha,a} (\beta). \]
 Now we prove 
\begin{equation}\label{var1}
\lim_{N \rightarrow \infty} {\rm Var}_{\nu_\rho^\star} \Big( \frac{1}{N^\alpha} \int_0^{tN^\alpha} \frac{1}{N^d} \sum_{z \in \Z^d_\star} p \big(\tfrac{z}{N}\big) (e^{i \beta (z \cdot a) /N} - 1 - i \beta (z \cdot a) /N) (1 - \xi_s (z)) ds \Big) = 0,
\end{equation}
which, by Chebyshev's inequality, implies that the term \eqref{eqn:1} converges  as $N \rightarrow \infty$  to $t (1-\rho) \Phi_{\alpha,a} (\beta)$ in $\P_{\nu_\rho^\star}$--probability. 
 For the real part, by Cauchy-Schwarz  inequality and stationary of the process $\{\xi_t\}_{t \geq 0}$, there exists some constant $C$ such that the variance  in \eqref{var1} is bounded by 
 \[
\frac{Ct^2}{N^{2d}}  \sum_{z \in \Z^d_\star} p \big(\tfrac{z}{N}\big)^2 \big[ \cos\big(\tfrac{\beta (z \cdot a)}{N}\big) - 1 \big]^2 
\leq \frac{Ct^2}{N^d} \Big( 1 + \int_{\tfrac{1}{N}}^1 r^{3-d-2\alpha} dr \Big) \leq C t^2 \big(N^{-d} + N^{2 \alpha - 4}\big).
\]
The last line converges to zero since $\alpha < 2$. Similarly, the variance of the imaginary part in \eqref{var1} is bounded by
\[\frac{Ct^2}{N^{2d}} \sum_{z \in \Z^d_\star} p \big(\tfrac{z}{N}\big)^2 \big[ \sin \big( \tfrac{\beta (z \cdot a)}{N} \big) - \tfrac{\beta (z \cdot a)}{N}\big]^2 
\leq \frac{Ct^2}{N^d} \Big( 1 + \int_{\tfrac 1 N}^1 r^{5-d-2\alpha} dr \Big) 
\leq Ct^2 \big( N^{-d} + N^{2 \alpha - 6} \big),\]
which also vanishes in the limit. This proves \eqref{var1}.  

\medspace

To finish the proof,  we only need to show \eqref{eqn:2} vanishes in the limit, precisely speaking,
\begin{equation}\label{eqn:3}
\lim_{N \rightarrow \infty} \sum_{z \in \Z^d_\star} (z \cdot a)  p(z) \frac{1}{N} \int_0^{tN^\alpha} (\rho - \xi_s (z))\,ds = 0
\end{equation}
in $\P_{\nu_\rho^\star}$-probability.   By Lemma \ref{lem:rw}, the random walk with transition kernel $s(\cdot)$ is transient if  $d \geq 2$ and $1 < \alpha < 2$, which permits us to exploit the \emph{transience estimates} from \cite{sethuraman2000diffusive}.  
By Kipnis-Varadhan inequality \cite[Proposition A1.6.1]{klscaling},
\begin{multline}\label{kv}
	\E_{\nu_\rho^\star} \Big[ \sup_{0 \leq t \leq T} \Big( \sum_{z \in \Z_\star^d} (z \cdot a)  p(z) \frac{1}{N} \int_0^{tN^\alpha} (\rho - \xi_s (z))\,ds \Big)^2  \Big] \\
	\leq 20 T N^{\alpha -2}  \sup_{f \in  L^2 (\Omega_\star^d,\nu^\star_\rho)} \Big\{  2 \int   \sum_{z \in \Z_\star^d} (z \cdot a)  p(z)  (\rho - \xi (z)) f (\xi)\,\nu_\rho^\star (d \xi) - \<f,(-\gen) f\>_{\nu_\rho^\star}\Big\}.
\end{multline}
Above, for a probability measure $\mu$ on $\Omega_\star^d$ and two functions $f,g \in L^2 (\Omega_\star^d,\mu)$, we write $\<f,g\>_\mu=E_\mu [fg]$.  We also remark that we need the supremum inside the above expectation in order to prove tightness in the next section. Recall $s(\cdot)$ is the normalized symmetric part of $p(\cdot)$,
\[s(z) = \frac{p(z)+p(-z)}{2 s^\star} =  \frac{1}{s^\star} \frac{1}{\|z\|^{d+\alpha}},  \quad z \in \Z^d_\star.\]
Since $\nu_\rho^\star$ is invariant for the generator $\gen$, 
\begin{multline*}
\<f,(-\gen) f\>_{\nu_\rho^\star} = \frac{s^\star}{2} \Big\{  \int  \sum_{x,y \in \Z^d_\star} s(y-x) [f(\xi^{x,y}) - f(\xi)]^2\,\nu_\rho^\star (d \xi)\\
+\int\sum_{z\in \Z^d_\star}  s(z) (1-\xi (z)) [f(\theta_z \xi) - f(\xi)]^2 \,\nu_\rho^\star (d \xi) \Big\}.
\end{multline*}
As in Appendix \ref{sec:tools}, for any finite subset $A \subset \Z^d_\star$, define 
\[\Psi_A (\xi) = \prod_{x \in A} \frac{\xi(x) - \rho}{\sqrt{\rho(1-\rho)}}\]
with the convention that $\Psi_{\emptyset} = 1$. 
Then $\{\Psi_A:\,A \subset \Z^d_\star\;\text{is finite}\}$ is an orthonormal basis of $L^2 (\Omega_\star^d,\nu_\rho^\star)$. This permits us to write any function $f \in L^2  (\Omega_\star^d,\nu_\rho^\star)$ as
\[f = \sum_{A \subset \Z^d_\star \atop A \; \text{finite}} \mf{f} (A) \Psi_A,\]
where $\mf{f} (A) = \int f(\xi) \Psi_A (\xi)  \nu_\rho^\star (d \xi)$ for finite subset $A \subset \Z_\star^d$. One could check directly that  $\Psi_A (\xi^{x,y}) = \Psi_{A\backslash \{x\} \cup \{y\}} (\xi)$ if $x \in A,\,y \notin A$. Therefore,
\begin{equation}\label{dirichletform}
\<f,(-\gen) f\>_{\nu_\rho^\star} \geq s^\star \sum_{A \subset \Z^d_\star \atop A \; \text{finite}}  \sum_{x \in A,\atop y \notin A} s(y-x) \big[\mf{f} (A \backslash \{x\} \cup \{y\}) - \mf{f} (A)\big]^2. 
\end{equation}

Now we deal with the first term inside the supremum in \eqref{kv}. By change of variables $\xi \mapsto \xi^z$,
\[\int  (\rho - \xi (z)) f (\xi)\,\nu_\rho^\star (d \xi)  = (1-\rho) \int \chi_{\{\xi(z)=1\}} \big( f(\xi^z) - f(\xi) \big) \,\nu_\rho^\star (d \xi),\]
where $\xi^z$ is the configuration obtained from $\xi$ by flipping the value of $\xi(z)$, \emph{i.e.}, $\xi^z (x) = \xi (x)$ for $x \neq z$ and $\xi^z (z) = 1- \xi (z)$. By Cauchy-Schwarz inequality, we may bound the first term inside the supremum in \eqref{kv} from above by
\begin{multline}\label{glauber}
	2 (1-\rho) \int \sum_{z \in \Z_\star^d} (z \cdot a)  p(z) \chi_{\{\xi(z)=1\}}  \big( f(\xi^z) - f(\xi) \big) \,\nu_\rho^\star (d \xi)\\
	\leq 2 (1-\rho) \Big(\sum_{z \in \Z_\star^d} |z \cdot a|  p(z)\Big)^{1/2} \Big( \int \sum_{z \in \Z_\star^d} |z \cdot a|  p(z) \big( f(\xi^z) - f(\xi) \big)^2  \,\nu_\rho^\star (d \xi)\Big)^{1/2}\\
\leq  C \Big( \int \sum_{z \in \Z_\star^d} |z \cdot a|  p(z) \big( f(\xi^z) - f(\xi) \big)^2  \,\nu_\rho^\star (d \xi)\Big)^{1/2}
\end{multline}
for some constant $C = C(\rho,\alpha,a)$. Since
\[\Psi_A (\xi^z) - \Psi_A (\xi) = \begin{cases}
	0, \quad&z \notin A,\\
	\frac{1-2\xi(z)}{\sqrt{\rho(1-\rho)}}\Psi_{A \backslash z} (\xi), \quad&z \in A,
\end{cases}\]
direct calculations show that
\begin{multline*}
\int \big( f(\xi^z) - f(\xi) \big)^2 \,\nu_\rho^\star (d \xi) = \int  \Big( \sum_{A \subset \Z^d_\star \atop A \; \text{finite}} \mf{f} (A) [\Psi_A (\xi^z) - \Psi_A (\xi)] \Big)^2 \,\nu_\rho^\star (d \xi) \\
= \frac{1}{\rho(1-\rho)} \int  \Big( \sum_{A:z \in A} \mf{f} (A) \Psi_{A\backslash z} (\xi) \Big)^2 \,\nu_\rho^\star (d \xi) = \frac{1}{\rho(1-\rho)} \sum_{A:z \in A} \mf{f}^2 (A).
\end{multline*}
For $n \geq 0$, let 
\[F_n^2 (z) = \sum_{A:z \in A, \atop |A|=n} \mf{f}^2 (A).\]
Then we may rewrite the last line in \eqref{glauber} as 
\begin{equation}\label{glauber_1}
C \Big( \sum_{z \in \Z_\star^d} \sum_{n \geq 0} |z \cdot a|  p(z) F_n^2 (z) \Big)^{1/2}
\end{equation}
for some constant $C = C(\rho,\alpha,a)$.

Next we bound $F_n^2 (z)$. Denote by $S_t (x,y)$ the transition probability associated with the continuous time random walk with transition probability kernel $s(\cdot)$ and by $g(x,y)$ the corresponding Green's function,
\[g(x,y) = \int_0^\infty S_t (x,y)\,dt.\]
By Lemma \ref{lem:rw},  $g(x,x) < \infty$ if $d \geq 2$ and $1 < \alpha < 2$.  Moreover, it is easy to see
\begin{equation}\label{green}
	\sum_{y} s(x-y) [g(z,y) - g (z,x) ] = - \chi_{\{x=z\}}.
\end{equation}
Indeed, since 
\[\frac{d S_t (x,z)}{dt} = \sum_{y} s(x-y) \big[S_t (y,z) - S_t (x,z)\big],\]
integrating over time from zero to infinity, we obtain \eqref{green} by the symmetry of $S_t$. This permits us to rewrite $F_n^2 (z)$ as
\begin{multline*}
F_n^2 (z)  = g(z,z) \sum_{x} \frac{F_n^2 (x)}{g(z,x)} \chi_{\{x=z\}} = -  g(z,z) \sum_{x,y}  \frac{F_n^2 (x)}{g(z,x)}  s(y-x) [g(z,y) - g (z,x) ]  \\
= (1/2) g(z,z) \sum_{x,y} s(y-x) \Big[\frac{F_n^2 (y)}{g(z,y)}  -  \frac{F_n^2 (x)}{g(z,x)} \Big]  [g(z,y) - g (z,x) ].
\end{multline*}
Above, in the second identity we use \eqref{green} and in the last one we use the symmetry of $s(\cdot)$. By Cauchy-Schwarz inequality,
\begin{multline*}
\Big[\frac{F_n^2 (y)}{g(z,y)}  -  \frac{F_n^2 (x)}{g(z,x)} \Big]  [g(z,y) - g (z,x) ] = F_n^2 (y) + F_n^2 (x) - \Big[ \frac{g(z,y)}{g(z,x)} F_n^2 (x) + \frac{g(z,x)}{g(z,y)}F_n^2 (y) \Big]\\
\leq  \Big(F_n (y) - F_n (x)\Big)^2,
\end{multline*}
Therefore,
\[F_n^2 (z) \leq (1/2) g(z,z) \sum_{x,y} s(y-x) \Big(F_n (y) - F_n (x)\Big)^2.\]
Since $g(z,z) = g(0,0)$ is independent of $z$ and the jump rate $p(\cdot)$ has finite first moment if $\alpha >1$, we bound \eqref{glauber_1} from above by 
\begin{equation}\label{rw_dirichlet}
C \Big( \sum_{n} \sum_{x,y} s(y-x) \big(F_n (y) - F_n (x)\big)^2 \Big)^{1/2}.
\end{equation}
for some constant $C = C(\rho,\alpha,a)$. By the definition of $F_n$ and Cauchy-Schwarz inequality,
\begin{multline*}
| F_n^2 (y) - F_n^2 (x) |= \Big| \sum_{x \in A, y \notin A, \atop |A| = n} \big[\mf{f}^2 (A \backslash \{x\} \cup \{y\}) - \mf{f}^2 (A)\big]\Big| \\
\leq \sqrt{\sum_{x \in A, y \notin A, \atop |A| = n} \big[\mf{f} (A \backslash \{x\} \cup \{y\}) - \mf{f} (A)\big]^2} \times \sqrt{\sum_{x \in A, y \notin A, \atop |A| = n} \big[\mf{f} (A \backslash \{x\} \cup \{y\}) + \mf{f} (A)\big]^2} \\
\leq  \sqrt{\sum_{x \in A, y \notin A, \atop |A| = n} \big[\mf{f} (A \backslash \{x \}\cup \{y\}) - \mf{f} (A)\big]^2} \times \big(F_n(y) + F_n (x)\big).
\end{multline*}
Thus,
\[ \big(F_n (y) - F_n (x)\big)^2 \leq \sum_{A: x \in A, y \notin A, \atop |A| = n} \big(\mf{f} (A \backslash \{x\} \cup \{y\}) - \mf{f} (A)\big)^2.\]
This permits us to bound \eqref{rw_dirichlet} by
\[C \Big( \sum_{A \subset \Z^d_\star \atop A \; \text{finite}}  \sum_{x \in A,\atop y \notin A} s(y-x) \big(\mf{f} (A \backslash \{x \}\cup \{y\}) - \mf{f} (A)\big)^2 \Big)^{1/2} \leq C \<f,(-\gen)f\>_{\nu_\rho^\star}^{1/2}.\]
for some constant $C = C(\rho,\alpha,a)$. In the last inequality we use \eqref{dirichletform}. Since $Cx - x^2 \leq C^2/4$, we bound \eqref{kv} from above by  $C TN^{\alpha -2}$, which converges to zero since $\alpha < 2$. This concludes the proof for the case  $d \geq 2$ and $1 < \alpha < 2$.
\end{proof}

\begin{remark}
	Using the techniques presented in the following cases, one could also prove that 
	\[\E_{\nu_\rho^\star} \Big[ \sup_{0 \leq t \leq T} \Big( \int_0^{tN^\alpha} (\rho - \xi_s (z) )ds \Big)^2 \Big] \leq C N^\alpha\]
	uniformly in $z \in \Z_\star^d$.  Then, \eqref{eqn:3} follows directly from Cauchy-Schwarz inequality.
\end{remark}

\subsection{The case $d=1$ and $1< \alpha < 3/2$.}  Following the proof in the case $d \geq 2$ and $1< \alpha < 2$ line by line, we only need to prove \eqref{eqn:3}. By Cauchy-Schwarz inequality, 
\begin{multline}\label{oc2}
\E_{\nu_\rho^\star} \bigg[ \sup_{0 \leq t \leq T} \Big( \sum_{z \in \Z_\star} (z \cdot a)  p(z) \frac{1}{N} \int_0^{tN^\alpha} (\rho - \xi_s (z))\,ds \Big)^2  \bigg] \\
\leq C N^{-2} \sum_{z \in \Z_\star} |z \cdot a|  p(z) \E_{\nu_\rho^\star} \bigg[ \sup_{0 \leq t \leq T} \Big(  \int_0^{tN^\alpha} (\rho - \xi_s (z))\,ds \Big)^2  \bigg]
\end{multline}
for some constant $C = C(\alpha,a)$. To conclude the proof, we only need to show that there exists some constant $C$  such that
\[\sup_{z \in \Z_\star^d}\E_{\nu_\rho^\star} \Big[ \sup_{0 \leq t \leq T}  \Big(  \int_0^{tN^\alpha} (\rho - \xi_s (z))\,ds \Big)^2  \Big] \leq C N^{2\alpha-1},\]
or equivalently, for any $t > 0$,
\begin{equation}\label{oc1}
\sup_{z \in \Z_\star^d} \E_{\nu_\rho^\star} \Big[ \sup_{0 \leq t^\prime \leq t}  \Big(  \int_0^{t^\prime} (\rho - \xi_s (z))\,ds \Big)^2  \Big] \leq C t^{2-\tfrac{1}{\alpha}}.
\end{equation}
By \cite[Lemma 3.9]{sethuraman2000central}, the expectation in \eqref{oc1} is bounded by 
\[10 t \, \<\xi(z) - \rho,(1/t - \gen)^{-1} (\xi(z) - \rho)\>_{\nu_\rho^\star}.  \]
As in Appendix \ref{sec:tools}, let $\gens :=  (\gen + \gen^*)/2$ and $\gena := (\gen - \gen^*)/2$ be the symmetric and anti-symmetric part of the generator $\gen$ in $L^2 (\Omega_\star^d,\nu_\rho^\star)$ respectively.  By \cite[Eq.\,(3.3)]{sethuraman2000central}, for any local function $f: \Omega_\star^d \rightarrow \R$ and for any $\lambda > 0$,
\[ \<f ,(\lambda - \gen)^{-1}f\>_{\nu_\rho^\star}= \sup_{g\;\text{local}} \Big\{ 2 \<f,g\>_{\nu_\rho^\star} - \<g,(\lambda - \gens)g\>_{\nu_\rho^\star} - \<\gena g, (\lambda - \gens)^{-1} \gena g\>_{\nu_\rho^\star}\Big\}.\]
Since the generator $\gens$ is reversible with respect to $\nu_\rho^\star$, we have $\<\gena g, (\lambda - \gens)^{-1} \gena g\>_{\nu_\rho^\star} \geq 0$. Thus,
\[\<f ,(\lambda - \gen)^{-1}f\>_{\nu_\rho^\star} \leq \sup_{g\;\text{local}} \Big\{ 2 \<f,g\>_{\nu_\rho^\star} - \<g,(\lambda - \gens)g\>_{\nu_\rho^\star} \Big\} =  \|f\|_{-1,\lambda,\gens}.\]
We refer the readers to Appendix \ref{sec:tools} for precise definitions of the norm $\|\cdot\|_{-1,\lambda,\gens}$. In particular,
\[\<\xi(z) - \rho,(1/t - \gen)^{-1} (\xi(z) - \rho)\>_{\nu_\rho^\star} \leq \|\xi(z) - \rho\|_{-1,t^{-1},\mathcal{S}} = \<\xi(z) - \rho,(1/t - \gens)^{-1} (\xi(z) - \rho)\>_{\nu_\rho^\star}.\]
Since the function $\xi(z) - \rho$ has degree one, by Lemmas \ref{lem-a1} and \ref{lem-a2},  the last line is bounded by $C \|\chi_{\{z\}}\|_{-1,t^{-1},\mf{S}_{\rm ext}}$ for some constant $C = C(\alpha,\rho)$, where the function $\chi_{\{z\}}: \bar{\mc{E}}_{d,1} \rightarrow \R$ is defined as
\[\chi_{\{z\}} (\{x\}) = \begin{cases}
1, \quad\text{if } x = z;\\
0,\quad\text{otherwise},
\end{cases}\]
and $ \bar{\mc{E}}_{d,1}$ is the family of subsets of $\Z^d$ with only one element.  The generator $\mf{S}_{\rm ext}$ is defined in Appendix \ref{appen1}.  We only need to note that  when acting on degree one functions, it corresponds to random walk with transition probability kernel $s(\cdot)$.  The quantity $\|\chi_{\{z\}}\|_{-1,t^{-1},\mf{S}_{\rm ext}}$ is closely related to the occupation time of the symmetric exclusion process on $Z^d$ with transition probability $s(\cdot)$. Precisely speaking, denote by $\{\hat{\eta}_t\}_{t \geq 0}$ the exclusion process with generator $\mathbb{S}$, which acts on local functions $f: \Omega^d \rightarrow \R$ as
\[\mathbb{S} f (\eta) = \frac{1}{2} \sum_{x,y \in \Z^d} s(y-x) [f(\eta^{x,y}) - f(\eta)].\]
Then, one could prove that (see \cite[Eq.\,(3.2)]{bernardin2016occupation} for example)
\begin{align*}
\|\chi_{\{z\}}\|_{-1,t^{-1},\mf{S}_{\rm ext}} &= \frac{1}{\rho(1-\rho)} \<\hat{\eta}(z)-\rho, (t^{-1}-\mathbb{S})^{-1}  (\hat{\eta}(z)-\rho) \>_{\nu_\rho} \\
&= \frac{1}{2\rho(1-\rho)t^2} \int_0^\infty e^{-s/t}  \E_{\nu_\rho} \Big[ \Big(\int_0^s (\hat{\eta}_\tau (z) - \rho) d \tau\Big)^2\Big] ds.
\end{align*}
Note that the above quantity is independent of $z$ because of translation invariance. By \cite[Theorem 2.8]{bernardin2016occupation}, if $d = 1$ and $1 < \alpha < 2$,
\[\E_{\nu_\rho} \Big[ \Big(\int_0^s (\hat{\eta}_\tau (z) - \rho) d \tau\Big)^2 \leq C s^{2-1/\alpha}.\]
Therefore,
\[\|\chi_{\{z\}}\|_{-1,t^{-1},\mf{S}_{\rm ext}} \leq C t^{-2} \int_0^\infty e^{-s/t} s^{2-1/\alpha} ds \leq C t^{1-1/\alpha}\]
for some constant $C = C(\alpha,\rho)$. This proves \eqref{oc1} and completes the proof in the case $d=1$ and $1< \alpha < 3/2$.

\subsection{The case $\alpha = 1$.}  Recall in this case $\overline{X}^N_{t} = X_{tN} - tN(1-\rho) \sum_{\|z\|\leq N} z p(z)$. Take $\gamma (N) = N$ and we rewrite the martingale in \eqref{mn} as
\begin{multline*}
		\mc{M}_{t}^N (\beta) = \exp \Big\{i \beta (\overline{X}^N_{t} \cdot a) /N  - i \beta \sum_{\|z\| \leq N} (z \cdot a) p(z) \frac{1}{N} \int_0^{tN} (\rho-\xi_s (z)) ds\\
		 -  \sum_{z \in \Z^d_\star} \big(e^{i\beta (z \cdot a)/N} - 1 - i\beta N^{-1}(z \cdot a) \chi_{\{\|z\| \leq N\}}\big) p(z) \int_0^{t N} (1-\xi_s(z))\,ds\Big\}.
\end{multline*}
As in the proof of the previous cases, by dominated convergence theorem, we only need to prove 
\begin{equation}\label{cri1}
	\lim_{N \rightarrow \infty} \sum_{\|z\| \leq N} (z \cdot a) p(z) \frac{1}{N} \int_0^{tN} (\rho-\xi_s (z)) ds = 0
\end{equation}
and
\begin{equation}\label{cri2}
	\lim_{N \rightarrow \infty} \sum_{z \in \Z^d_\star} \big(e^{i\beta (z \cdot a)/N} - 1 - i\beta N^{-1}(z \cdot a) \chi_{\{\|z\| \leq N\}}\big) p(z) \int_0^{t N} (1-\xi_s(z))\,ds = \Phi_{\alpha,a} (\beta)
\end{equation}
in $\P_{\nu_\rho^\star}$-probability. 

We first prove \eqref{cri1}. By Cauchy-Schwarz inequality,
\begin{multline*}
\E_{\nu_\rho^\star} \Big[ \sup_{0 \leq t \leq T} \Big(\sum_{\|z\| \leq N} (z \cdot a) p(z) \frac{1}{N} \int_0^{tN} (\rho-\xi_s (z)) ds\Big)^2\Big]\\
 \leq N^{-2}  \Big(\sum_{\|z\| \leq N} |z \cdot a| p(z) \Big) \sum_{\|z\| \leq N} |z \cdot a|  p(z) \E_{\nu_\rho^\star} \Big[ \sup_{0 \leq t \leq T}  \Big( \int_0^{tN} (\rho-\xi_s (z)) ds\Big)^2 \Big] \\
 \leq C N^{-2} (\log N)^2 \sup_{\|z\| \leq N} \E_{\nu_\rho^\star} \Big[  \sup_{0 \leq t \leq T}  \Big( \int_0^{tN} (\rho-\xi_s (z)) ds\Big)^2 \Big].
\end{multline*}
As in the proof of the case $d=1$ and $1< \alpha < 3/2$,  there exists some constant $C = C(\alpha,\rho)$ such that
\[\E_{\nu_\rho^\star} \Big[ \sup_{0 \leq t^\prime \leq t}  \Big( \int_0^{t^\prime} (\rho-\xi_s (z)) ds\Big)^2 \Big] \leq C t^{-1} \int_0^\infty e^{-s/t}  \E_{\nu_\rho} \Big[ \Big(\int_0^s (\hat{\eta}_\tau (z) - \rho) d \tau\Big)^2\Big] ds.\]
By \cite[Theorem 2.8]{bernardin2016occupation}, if $d = 1$ and $\alpha = 1$,
\[\E_{\nu_\rho} \Big[ \Big(\int_0^s (\hat{\eta}_\tau (z) - \rho) d \tau\Big)^2\Big] \leq C s \log s,\]
and if $d \geq 2$ and $\alpha = 1$,
\[\E_{\nu_\rho} \Big[ \Big(\int_0^s (\hat{\eta}_\tau (z) - \rho) d \tau\Big)^2\Big] \leq C s.\]
Therefore,
\begin{equation}\label{cri3}
\E_{\nu_\rho^\star} \Big[ \sup_{0 \leq t^\prime \leq t}  \Big( \int_0^{t^\prime} (\rho-\xi_s (z)) ds\Big)^2 \Big]  \leq C t^{-1} \int_0^\infty e^{-s/t}  \big(s \log s + s\big) ds
\leq C t (1+\log t).
\end{equation}
In particular,
\[\sup_{\|z\| \leq N} \E_{\nu_\rho^\star} \Big[  \sup_{0 \leq t \leq T}  \Big( \int_0^{tN} (\rho-\xi_s (z)) ds\Big)^2 \Big] \leq C N \log N,\]
and thus
\[\E_{\nu_\rho^\star} \Big[ \sup_{0 \leq t \leq T} \Big(\sum_{\|z\| \leq N} (z \cdot a) p(z) \frac{1}{N} \int_0^{tN} (\rho-\xi_s (z)) ds\Big)^2\Big] \leq C N^{-1} (\log N)^3 \]
for some constant $C = C(\rho,\beta,a,T)$, which concludes the proof of \eqref{cri1}.

Now, we prove \eqref{cri2}.  One could check directly that the expectation on the right side of \eqref{cri2} converges to $\Phi_{\alpha,a} (\beta)$ as $N \rightarrow \infty$. Then, we only need to prove its variance vanishes,
\[\lim_{N \rightarrow \infty} {\rm Var}_{\nu_\rho^\star} \Big( \sum_{z \in \Z^d_\star} \big(e^{i\beta (z \cdot a)/N} - 1 - i\beta N^{-1}(z \cdot a) \chi_{\{\|z\| \leq N\}}\big) p(z) \int_0^{t N} (1-\xi_s(z))\,ds \Big) = 0.\]
By Cauchy-Schwarz inequality and \eqref{cri3}, we could bound the variance in the last line by
\begin{multline*}
\Big( \sum_{z \in \Z^d_\star} \big|e^{i\beta (z \cdot a)/N} - 1 - i\beta N^{-1}(z \cdot a) \chi_{\{\|z\| \leq N\}}\big| p(z) \Big)^2 \sup_{z \in \Z_\star^d}  \E_{\nu_\rho^\star} \Big[  \sup_{0 \leq t \leq T}\Big(\int_0^{tN} (\xi_s(z) - \rho) d s \Big)^2\Big] \\
\leq C N^{-1} \log N
\end{multline*}
for some constant $C = C(\rho,\beta,a,T)$. This proves \eqref{cri2} and completes the proof in the case $ \alpha =1$.

\subsection{The case $d\geq2$ and $\alpha = 2$.} Recall in this case $\overline{X}^N_{t}$ has a $\log N$ correction \[\overline{X}^N_{t} = X_{tN^2/\log N} - t(N^2/\log N)(1-\rho)m.\]  Take $\gamma(N) = N^2 / \log N$ in \eqref{mn},  then we could rewrite the martingale $\mc{M}_{t}^N (\beta)$ as
\begin{multline}\label{m2}
	\mc{M}_{t}^N (\beta) = \exp \Big\{i \beta (\overline{X}^N_{t} \cdot a) /N  - i \beta \sum_{z \in \Z^d_\star} (z \cdot a) p(z) \frac{1}{N} \int_0^{tN^2/\log N} (\rho-\xi_s (z)) ds\\
	-  \sum_{z \in \Z^d_\star} \big(e^{i\beta (z \cdot a)/N} - 1 - i\beta N^{-1}(z \cdot a) \big) p(z) \int_0^{t N^2/\log N} (1-\xi_s(z))\,ds\Big\}.
\end{multline}
Then,  we only need to prove that
\begin{equation}\label{cri4}
\lim_{N \rightarrow \infty} \sum_{z \in \Z^d_\star} (z \cdot a) p(z) \frac{1}{N} \int_0^{tN^2/\log N} (\rho-\xi_s (z)) ds = 0
\end{equation}
and
\begin{equation}\label{cri5}
	\lim_{N \rightarrow \infty} \sum_{z \in \Z^d_\star} \big(e^{i\beta (z \cdot a)/N} - 1 - i\beta N^{-1}(z \cdot a) \big) p(z) \int_0^{t N^2/\log N} (1-\xi_s(z))\,ds = \Phi_{\alpha,a} (\beta)
\end{equation}
in $\P_{\nu_\rho^\star}$-probability. 

The proof of \eqref{cri4} is similar to that of \eqref{cri1}.  We first use Cauchy-Schwarz inequality and bound the variance of it by
\[C N^{-2} \sup_{z \in \Z_\star^d} \E_{\nu_\rho^\star} \Big[ \sup_{0 \leq t \leq T} \Big( \int_0^{tN^2/\log N} (\rho-\xi_s (z)) ds \Big)^2\Big].\]
By \cite[Theorem 2.8]{bernardin2016occupation}, if $d = 2$ and $\alpha = 2$,
\[\E_{\nu_\rho} \Big[ \Big(\int_0^s (\hat{\eta}_\tau (z) - \rho) d \tau\Big)^2\Big] \leq C s \log \log s,\]
and  if $d \geq 3$ and $\alpha = 2$,
\[\E_{\nu_\rho} \Big[ \Big(\int_0^s (\hat{\eta}_\tau (z) - \rho) d \tau\Big)^2\Big] \leq C s.\]
Therefore,
\begin{multline}\label{cri7}
\E_{\nu_\rho^\star} \Big[ \sup_{0 \leq t^\prime \leq t} \Big( \int_0^{t^\prime} (\rho-\xi_s (z)) ds \Big)^2\Big] \leq Ct^{-1} \int_0^\infty e^{-s/t} \E_{\nu_\rho} \Big[ \Big(\int_0^s (\hat{\eta}_\tau (z) - \rho) d \tau\Big)^2\Big] ds\\
\leq  \begin{cases}
C t (\log \log t + 1) \quad &\text{if } d = 2, \alpha =2,\\
C t  \quad &\text{if } d \geq 3, \alpha =2.
\end{cases}
\end{multline}
In particular, the variance of the term on the right side of \eqref{cri4} is bounded by $C \log \log N / \log N$ if $d = 2$ and $\alpha = 2$, and by $C / \log N$ if $d \geq 3$ and $\alpha = 2$.  Note that the expectation of the term on the right side of \eqref{cri4} is zero. This proves \eqref{cri4}.

It remains to prove \eqref{cri5}. Since
\begin{multline*}
\Big| \sum_{\|z\| > N} \big(e^{i\beta (z \cdot a)/N} - 1 - i\beta N^{-1}(z \cdot a) \big) p(z) \int_0^{t N^2/\log N} (1-\xi_s(z))\,ds \Big| \\
\leq \frac{Ct}{\log N} \int_1^\infty \frac{r}{r^{d+2}} r^{d-1} dr \leq \frac{Ct}{\log N}
\end{multline*}
and
\begin{multline*}
	\Big| \sum_{\|z\| \leq  N} \Big(e^{i\beta (z \cdot a)/N} - 1 - i\beta N^{-1}(z \cdot a) + \frac{\beta^2 (z \cdot a)^2}{2N^2} \Big) p(z) \int_0^{t N^2/\log N} (1-\xi_s(z))\,ds \Big| \\
	\leq \frac{Ct}{\log N} \int_{\tfrac 1 N}^1 \frac{r^3}{r^{d+2}} r^{d-1} dr \leq \frac{Ct}{\log N},
\end{multline*}
we only need to prove
\begin{equation}\label{cri6}
	\lim_{N \rightarrow \infty} \frac{\beta^2}{2} \sum_{\|z\| \leq  N} \Big(\frac{z\cdot a}{N}\Big)^2 p(z) \int_0^{t N^2/\log N} (1-\xi_s(z))\,ds =- \Phi_{\alpha,a} (\beta)\end{equation}
in $\P_{\nu_\rho^\star}$-probability.  It is easy to see the expectation of the term on the right side of the last line converges to $\Phi_{\alpha,a} (\beta)$. Therefore, we only need to prove its variance converges to zero. By Cauchy-Schwarz inequality and \eqref{cri7}, we bound the variance by
\begin{multline*}
	C \Big[ \sum_{\|z\| \leq  N} \Big(\frac{z\cdot a}{N}\Big)^2 p(z) \Big]^2 \sup_{\|z\| \leq N} \E_{\nu_\rho^\star} \Big[\sup_{0 \leq t \leq T} \Big( \int_0^{t N^2/\log N} (\rho-\xi_s(z))\,ds \Big)^2\Big] \\
	\leq  \begin{cases}
		C (\log N)^2 / N^2 \quad &\text{if } d = 2, \alpha =2,\\
		C \log N / N^2  \quad &\text{if } d \geq 3, \alpha =2.
	\end{cases}
\end{multline*}
This proves \eqref{cri5} and concludes the proof in the case $d = 2$ and $\alpha \geq 2$.

\section{Tightness}\label{sec:tight} For $N \geq 1$, denote by $Q^N$ the distribution on the Skorohod space $D([0,T],\R^d)$ of the process $\big\{N^{-1} \overline{X}^N_t\big\}_{0 \leq t \leq T}$. In this section, we prove the tightness of the sequence of distributions $\{Q^N\}_{N \geq 1}$.

\begin{lemma}[Tightness]\label{lem:tight}
Under the assumptions in Theorem \ref{thm:invariant}, the sequence of distributions $\{Q^N\}_{N \geq 1}$ is tight in the Skorohod space $D([0,T],\R^d)$.
\end{lemma}

By Prohorov's theorem and Aldous' criterion (\emph{cf.}\,\cite[Section 4.1]{klscaling} for example), we only need to check the following two conditions in order to prove tightness.

\begin{proposition}
The sequence of the distributions $\{Q^N\}_{N \geq 1}$ is tight if 
\begin{enumerate}[(i)]
	\item for any $0 \leq t \leq T$, 
	\begin{equation}\label{tight_1}
	\lim_{M \rightarrow \infty} \limsup_{N \rightarrow \infty} Q^N \big(|\omega (t)| > M\big) = 0;
	\end{equation}
    \item for any $\varepsilon > 0$, 
    \begin{equation}\label{tight_2}
    	\lim_{\theta \rightarrow 0} \limsup_{N \rightarrow \infty} \sup_{\tau \in \mathfrak{T}_T\atop \delta \leq \theta} Q^N \big(|\omega (\tau+\delta) -\omega(\tau)| > \varepsilon\big) = 0,
    \end{equation}
where $\mathfrak{T}_T$ is the set of stopping times bounded from above by $T$.
\end{enumerate}
\end{proposition}

One could check directly that the mapping $\beta \mapsto \Phi_{\alpha,a} (\beta)$   is continuous at zero. By Continuity theorem (\emph{cf}.\,\cite[Theorem 3.3.17]{durrett2017probability} for example), for any fixed $0 \leq t \leq T$, the sequence of random elements $\big\{N^{-1} \overline{X}^N_t\big\}_{N \geq 1}$ is tight. Therefore, condition \eqref{tight_1} is satisfied.  In the rest of this section, we prove condition \eqref{tight_2} in different regimes of $d$ and $\alpha$. 

\subsection{The case $0 < \alpha <1$.}\label{subsec:5.1}  For any stopping time $\tau \leq T$, any $\delta \leq \theta$, and  any $\varepsilon > 0$,
\begin{multline*}
	\P_{\nu_\rho^\star} \Big( \Big| \frac{(\overline{X}^N_{\tau+\delta} -\overline{X}^N_{\tau}) \cdot a}{N} \Big| > \varepsilon \Big) \leq \frac{\varepsilon}{2} \int_{-2/ \varepsilon}^{2/\varepsilon} \Big|1 - \E_{\nu_\rho^\star} \Big[ \exp \Big\{ i \beta \frac{(\overline{X}^N_{\tau+\delta} -\overline{X}^N_{\tau}) \cdot a}{N}  \Big\} \Big]\Big| d\beta\\
	\leq \frac{\varepsilon}{2} \int_{-2/ \varepsilon}^{2/\varepsilon}  \E_{\nu_\rho^\star} \Big[ \Big| 1 - \exp \Big\{ - \int_{\tau N^\alpha}^{(\tau+\delta)N^\alpha}  \sum_{z \in \Z^d_\star} \big(e^{i\beta(z\cdot a)/N} -1\big) p(z) (1-\xi_s (z)) ds  \Big\} \Big| \Big] d\beta.
\end{multline*}
Above, the second inequality follows from the fact that $\mathcal{M}^N_{\tau+\cdot} (\beta) / \mathcal{M}^N_{\tau} (\beta)$ is a mean one complex martingale with $\mathcal{M}^{N}_{\cdot} (\beta)$ defined in \eqref{mn}.  Since for any $|\beta| \leq 2/ \varepsilon$,
\[\Big|  N^\alpha \sum_{z \in \Z^d_\star} \big(e^{i\beta(z\cdot a)/N} -1\big) p(z) (1-\xi (z)) \Big| \leq C \]
for some constant $C = C(a,\varepsilon)$, and $|e^z - 1| \leq |z| e^{|z|}$ for any $z \in \mathbb{C}$, we have
\[\P_{\nu_\rho^\star} \Big( \Big| \frac{(\overline{X}^N_{\tau+\delta} -\overline{X}^N_{\tau}) \cdot a}{N} \Big| > \varepsilon \Big)\leq C\delta e^{C \delta} \leq C \theta e^{C \theta}. \]
This  proves \eqref{tight_2}  if $0 < \alpha <1$.

\subsection{The case $d \geq 2, 1 < \alpha <2$ and $d = 1, 1 < \alpha < 3/2$.}\label{subsec:5.2} To prove \eqref{tight_2} in this case, we need a truncation argument.  Let
\[\overline{X}^N_t := \overline{U}_{t}^N +  \overline{R}_t^N,\]
where
\[\overline{U}_{t}^N = \sum_{\|z\| \leq N} z \big[N^z_{tN^\alpha} - t N^\alpha(1-\rho) p(z)\big],\quad \overline{R}_t^N = \sum_{\|z\| > N} z \big[N^z_{tN^\alpha} - t N^\alpha(1-\rho) p(z)\big].\]

We first deal with the term $\overline{R}_t^N$, which is similar to the case $0 < \alpha <1$.  For any stopping time $\tau \leq T$, any $\delta \leq \theta$, and  any $\varepsilon > 0$,
\begin{multline*}
	\P_{\nu_\rho^\star} \Big( \Big| \frac{\big(\overline{R}_{\tau+\delta}^N-\overline{R}_{\tau}^N\big) \cdot a}{N} \Big| > \varepsilon \Big) \leq \frac{\varepsilon}{2} \int_{-2/ \varepsilon}^{2/\varepsilon} \Big|1 - \E_{\nu_\rho^\star} \Big[ \exp \Big\{ i \beta \frac{\big(\overline{R}_{\tau+\delta}^N-\overline{R}_{\tau}^N \big)\cdot a}{N}  \Big\} \Big]\Big| d\beta\\
	\leq \frac{\varepsilon}{2} \int_{-2/ \varepsilon}^{2/\varepsilon}  \E_{\nu_\rho^\star} \Big[ \Big| 1 - \exp \Big\{\frac{1}{N^\alpha}\int_{\tau N^\alpha}^{(\tau+\delta)N^\alpha} \widehat{\Gamma}_a^{N} (\xi_s)ds  \Big\} \Big| \Big] d\beta,
\end{multline*}
where $\widehat{\Gamma}_a^{N}$ is the truncation of $\Gamma^{N}_a $ defined in \eqref{GammaN} at the level $\|z\| > N$, 
\[	\widehat{\Gamma}_a^{N} (\xi) = \frac{1}{N^d} \sum_{\|z\|>N} p(\tfrac{z}{N}) (e^{i \beta (z \cdot a) /N} - 1 ) (1 - \xi (z)) 
- i \beta (1-\rho) N^{\alpha -1} \sum_{\|z\|>N} (z \cdot a)  p(z)  .\]
Since  $1 < \alpha < 2$,   there exists a constant $C=C(a,\varepsilon)$ such that for any $|\beta| \leq 2/\varepsilon$,
\[ |\widehat{\Gamma}_a^{N} (\xi)| \leq C   \int_{1}^\infty r^{-d-\alpha} (1+ r) r^{d-1} dr \leq  C.\]
Therefore,
\begin{equation}\label{tight1}
	\P_{\nu_\rho^\star} \Big( \Big| \frac{\big(\overline{R}_{\tau+\delta}^N-\overline{R}_{\tau}^N\big) \cdot a}{N} \Big| > \varepsilon \Big) \leq C\delta e^{C \delta} \leq C \theta e^{C \theta}.
\end{equation}

It remains to deal with the term $\overline{U}_t^N$. Since
\[N^z_{t} - \int_0^{t} p(z) (1-\xi_s (z))\,ds\]
is a martingale with quadratic variation  $\int_0^{t} p(z) (1-\xi_s (z))\,ds$, we may write $N^{-1} \overline{U}_{t}^N$ as
\begin{equation}\label{xNl}
N^{-1} \overline{U}_{t}^N = \mathfrak{m}^{N}_t + \frac{1}{N} \sum_{\|z\| \leq N} z p(z) \int_0^{tN^\alpha} (\rho - \xi_s (z))ds,
\end{equation}
where $\mathfrak{m}^{N}_t$ is a martingale with quadratic variation bounded by
\[\E_{\nu_\rho^\star} \big[ \big(\mathfrak{m}^{N}_t \cdot a\big)^2  \big] \leq \frac{Ct N^\alpha }{N^2} \sum_{\|z \| \leq N} |z \cdot a|^2 p(z) \leq Ct \int_{N^{-1}}^{1} r^{2-d-\alpha} r^{d-1} dr \leq C t.\] 
Therefore,
\begin{equation}\label{tight2}
\E_{\nu_\rho^\star} \big[\big\{  \big(\mathfrak{m}^{N}_{\tau+\delta} - \mathfrak{m}^{N}_{\tau}\big) \cdot a \big\}^2  \big] \leq C \delta \leq C \theta.
\end{equation}
Following the proof of \eqref{kv} and \eqref{oc2},  we have 
\begin{equation}\label{tight3}
\lim_{N \rightarrow \infty}\,\E_{\nu_\rho^\star} \Big[ \sup_{0 \leq t \leq T} \Big| \frac{1}{N} \sum_{\|z\| \leq N} (z \cdot a) p(z)  \int_0^{tN^\alpha} (\rho - \xi_s (z))\,ds \Big|^2  \Big] = 0.
\end{equation}
Since 
\begin{multline*}
\P_{\nu_\rho^\star} \Big( \Big| \frac{\big(\overline{U}_{\tau+\delta}^N-\overline{U}_{\tau}^N\big) \cdot a}{N} \Big| > \varepsilon \Big) \leq \P_{\nu_\rho^\star} \Big( \Big| \big(\mf{m}_{\tau+\delta}^N-\mf{m}_{\tau}^N\big) \cdot a \Big| > \varepsilon/2 \Big) \\
+ \P_{\nu_\rho^\star} \Big( \sup_{0 \leq t \leq T} \Big| \frac{1}{N} \sum_{\|z\| \leq N} (z \cdot a) p(z)  \int_0^{tN^\alpha} (\rho - \xi_s (z))\,ds \Big| > \varepsilon/4 \Big),
\end{multline*}
the term $N^{-1} \overline{U}_t^N$ also satisfies \eqref{tight_2} by  \eqref{tight2} and \eqref{tight3}.  This concludes the proof of Lemma \ref{lem:tight} in the case $d \geq 2, 1 < \alpha < 2$ and $d = 1, 1 < \alpha < 3/2$.
 
\subsection{The case $\alpha =1$.}\label{subsec:5.3} As in the last subsection, we decompose $\overline{X}^N_t := \overline{U}_{t}^N +  \overline{R}_t^N$, but with 
\[\overline{U}_{t}^N = \sum_{\|z\| \leq N} z \big[N^z_{tN} - t N (1-\rho) p(z)\big],\quad \overline{R}_t^N = \sum_{\|z\| > N} z N^z_{t} .\]
It is the same as in Subsection \ref{subsec:5.2} to check $N^{-1} \overline{U}_{t}^N $ satisfies \eqref{tight_2}.  For this reason, we omit the proof. To deal with the second term, for any stopping time $\tau \leq T$, any $\delta \leq \theta$, and  any $\varepsilon > 0$,
\begin{multline*}
	\P_{\nu_\rho^\star} \Big( \Big| \frac{\big(\overline{R}_{\tau+\delta}^N-\overline{R}_{\tau}^N\big) \cdot a}{N} \Big| > \varepsilon \Big) \\
	\leq \frac{\varepsilon}{2} \int_{-2/ \varepsilon}^{2/\varepsilon}  \E_{\nu_\rho^\star} \Big[ \Big| 1 - \exp \Big\{ - \int_{\tau N}^{(\tau+\delta)N}  \sum_{\|z\|>N} \big(e^{i\beta(z\cdot a)/N} -1\big) p(z) (1-\xi_s (z)) ds  \Big\} \Big| \Big] d\beta.
\end{multline*}
Since there exists a finite constant $C$ such that
\[\Big|N \sum_{\|z\|>N} \big(e^{i\beta(z\cdot a)/N} -1\big) p(z) (1-\xi(z)) \Big| \leq C,\]
we have
\[\P_{\nu_\rho^\star} \Big( \Big| \frac{\big(\overline{R}_{\tau+\delta}^N-\overline{R}_{\tau}^N\big) \cdot a}{N} \Big| > \varepsilon \Big) \leq C \theta e^{C \theta}.\]
This concludes the proof for the case $\alpha = 1$.

\subsection{The case $d\geq2$ and $\alpha = 2$.} In this case, we decompose $\overline{X}^N_t := \overline{U}_{t}^N +  \overline{R}_t^N$ with
\begin{align*}
\overline{U}_{t}^N = \sum_{\|z\| \leq N} z \big[N^z_{tN^2/\log N} - t (N^2/\log N)(1-\rho) p(z)\big],\\
\overline{R}_t^N = \sum_{\|z\| > N} z \big[N^z_{tN^2/\log N} - t (N^2/\log N)(1-\rho) p(z)\big].
\end{align*}
As in Subsection \ref{subsec:5.2} and by \eqref{m2},
\[
	\P_{\nu_\rho^\star} \Big( \Big| \frac{\big(\overline{R}_{\tau+\delta}^N-\overline{R}_{\tau}^N\big) \cdot a}{N} \Big| > \varepsilon \Big) 
	\leq \frac{\varepsilon}{2} \int_{-2/ \varepsilon}^{2/\varepsilon}  \E_{\nu_\rho^\star} \Big[ \Big| 1 - \exp \Big\{\int_{\tau N^2/\log N}^{(\tau+\delta)N^2/\log N} W_a^{N} (\xi_s)ds  \Big\} \Big| \Big] d\beta,
\]
where \[W_a^N (\xi) = i \beta N^{-1} \sum_{\|z\|>N} (z \cdot a) p(z) (\rho - \xi (z)) + \sum_{\|z\|>N} \big(e^{i\beta (z \cdot a)/N} - 1 - i\beta N^{-1}(z \cdot a) \big) p(z)(1-\xi(z)).\]
Since
\[\Big| \frac{N^2}{\log N} W_a^N (\xi)  \Big| \leq \frac{C}{\log N} \int_1^\infty r^{1-d-2}  r^{d-1} dr \leq \frac{C}{\log N}, \]
we have
\[\P_{\nu_\rho^\star} \Big( \Big| \frac{\big(\overline{R}_{\tau+\delta}^N-\overline{R}_{\tau}^N\big) \cdot a}{N} \Big| > \varepsilon \Big) \leq \frac{C \theta }{\log N}e^{C \theta / \log N}.\]
Thus, $N^{-1} \overline{R}_t^N$ satisfies condition \eqref{tight_2}. For $\overline{U}_{t}^N$, it remains to check the quadratic variation of the associated martingale $\mf{m}^N_t$  is bounded by $Ct$, where
\[N^{-1} \overline{U}_{t}^N = \mathfrak{m}^{N}_t + \frac{1}{N} \sum_{\|z\| \leq N} z p(z) \int_0^{tN^2/\log N} (\rho - \xi_s (z))ds.\]
Indeed, we bound it by
\[\frac{Ct}{\log N} \sum_{\|z\| \leq  N} (z \cdot a)^2 p(z) \leq Ct.\]
This concludes the proof of the tightness in  the case $d\geq2$ and $\alpha = 2$.

\appendix

\section{Tools}\label{sec:tools}

\subsection{$H_1$ and $H_{-1}$ norms.}\label{subsec:norms} Let $\{\omega_t\}_{t \geq 0}$ be a continuous-time Markov process on some complete and separable metric space $\Omega$ endowed with its Borel $\sigma$-algebra. Assume the process $\omega_t$ is reversible with respect to some probability measure $\pi$ on $\Omega$. Denote by $L$ the generator of the process $\omega_t$, and by $\mathcal{C}$ the core of the generator $L$.

For $\lambda > 0$ and functions $f \in L^2 (\Omega,\pi)$, define
\[\|f\|^2_{1,\lambda,L} = \<f,(\lambda - L) f\>_{\pi}, \quad \|f\|^2_{-1,\lambda,L} = \<f,(\lambda - L)^{-1} f\>_{\pi}.\]
It is well known that the norm $\|\,\cdot\,\|_{-1,\lambda,L}$ has the following variational formula,
\[\|f\|_{-1,\lambda,L}^2  = \sup_{g \in \mathcal{C}} \big\{ 2 \<f,g\>_{\pi} - \|g\|^2_{1,\lambda,L} \big\}.\]
We refer the readers to \cite{komorowski2012fluctuations} for a comprehensive study on the above norms.

\subsection{Comparison between different norms.} For $d \geq 1$, let $\mathcal{E}_d$ be the family of finite subsets of $\Z_\star^d$, and let $\mathcal{E}_{d,n}$ be those subsets of cardinality $n \geq 0$. Fix $\rho \in (0,1)$. For non-empty $A \in \mc{E}_d$, define
\[\Psi_A (\xi) = \prod_{x \in A} \frac{\xi(x) - \rho}{\sqrt{\rho(1-\rho)}}.\]
By convention, $\Psi_{\emptyset} = 1$. Then, $\{\Psi_A: A \in \mc{E}_d\}$ is a basis of the Hilbert space $L^2 \big( \Omega_\star^d, \nu_\rho^\star\big)$. Thus, for any function $f \in L^2 \big( \Omega_\star^d, \nu_\rho^\star\big)$, 
\[f (\xi) = \sum_{A \in \mc{E}_d} \mf{f} (A) \Psi_A (\xi)\]
for some coefficient function $\mf{f}: \mc{E}_d \rightarrow \R$. Moreover, for any functions $f,g \in L^2 \big( \Omega_\star^d, \nu_\rho^\star\big)$,
\[\<f,g\>_{{\nu_\rho^\star}} = \sum_{A \in \mc{E}_d} \mf{f}(A) \mf{g} (A) =: \<\mf{f},\mf{g}\>.\]
We say the  function $f$ and its coefficient function $\mf{f}$ are of degree $n$ if $f$ is in the span of $\{\Psi_A: \, A \in \mc{E}_{d,n}\}$, or equivalently, the support of $\mf{f}$ is contained in $\mc{E}_{d,n}$, 

We use $\gens$ to denote the symmetric part of the generator $\gen$ in $L^2 (\nu_\rho^\star)$, i.e., $\gens = (\gen+\gen^*)/2$. We could decompose $\gens=\gens^e + \gens^t$, where for local function  $f: \Omega_\star^d \rightarrow \R$,
\begin{align*}
	\gens^e f (\xi) &= \sum_{x,y \in \Z^d_\star} s(y-x) \xi(x) (1-\xi(y)) [f(\xi^{x,y}) - f(\xi)],\\
	\gens^t f(\xi) &= 	 \sum_{z\in \Z^d_\star}  s(z) (1-\xi (z)) [f(\theta_z \xi) - f(\xi)].
\end{align*}
Moreover, the operator $\gens =\gens^e + \gens^t $ has counterpart $\mf{S} = \mf{S}^e + \mf{S}^t$, where 
\[\gens^e f (\xi)= \sum_{A \in \mc{E}_d} (\mf{S}^e\mf{f}) (A) \Psi_A (\xi), \quad  \gens^t f (\xi)= \sum_{A \in \mc{E}_d} (\mf{S}^t \mf{f} ) (A) \Psi_A (\xi).\]
By \cite{sethuraman2000diffusive}, the operators $\mf{S}^e$ and  $\mf{S}^t$ have the following explicit expressions,
\begin{align*}
(\mf{S}^e\mf{f}) (A) =& \frac{1}{2} \sum_{x,y \in \Z_\star^d} s(y-x) [\mf{f}(A_{x,y}) - \mf{f}(A)],\\
(\mf{S}^t\mf{f}) (A) =& (1-\rho) \sum_{z \notin A, \atop z \in \Z_\star^d} s(z) [\mf{f} (\theta_{-z}A) - \mf{f}(A)]  + \rho \sum_{z \in A} s(z) [\mf{f} (\theta_{-z}A) - \mf{f}(A)]\\
&+ \sqrt{\rho(1-\rho)} \sum_{z \notin A, \atop z \in \Z_\star^d} s(z) [  \mf{f} (A \cup \{z\})-\mf{f} (\theta_{-z}(A \cup \{z\}))]\\
&+ \sqrt{\rho(1-\rho)} \sum_{z \in A} s(z) [\mf{f} (A \backslash \{z\}) - \mf{f}(\theta_{-z} (A \backslash \{z\}))].
\end{align*}
Above, for $A \subset \Z_\star^d$, 
\[A_{x,y} = \begin{cases}
	A \backslash \{x\} \cup \{y\}, \quad &x \in A,\,y \notin A;\\
	A \backslash \{y\} \cup \{x\}, \quad &x \notin A,\,y \in A;\\
	A \quad &\text{otherwise},
\end{cases}\qquad
\theta_x A = \begin{cases}
	A+x, \quad &-x \notin A;\\
	(A+x) \backslash \{0\} \cup \{x\}, \quad &-x \in A,
\end{cases}\]
where $A +x = \{y+x: y \in A\}$ if $A$ is nonempty, and $\emptyset + x = \emptyset$. Note that $0 \notin \theta_x A$.

With the above notations, direct calculations show that
\begin{equation}\label{dirichlet}
\<f, -\gens^e f\>_{\nu_\rho^\star} = \<\mf{f},-\mf{S}^e\mf{f}\> = \frac{1}{4} \sum_{x,y \in \Z_\star^d} \sum_{A \subset \Z^d_\star}s(y-x) [\mf{f}(A_{x,y}) - \mf{f}(A)]^2.
\end{equation}

The following result shows that the $H_1$ and $H_{-1}$ norms associated to the generators $\mc{S}$ and its environment part $\mc{S}^e$ are equivalent. 

\begin{lemma}\label{lem-a1}
	There exists a finite constant $C = C(d,n,\alpha)$ such that for any $\lambda > 0$ and local function $f$ with degree $n$,
	\[\|f\|_{1,\lambda,\gens^e} \leq \|f\|_{1,\lambda,\gens} \leq C \|f\|_{1,\lambda,\gens^e}.\]
	As a consequence,
	\[C^{-1} \|f\|_{-1,\lambda,\gens^e} \leq \|f\|_{-1,\lambda,\gens} \leq  \|f\|_{-1,\lambda,\gens^e}.\]
\end{lemma}

The above result was proved in \cite[Proposition 3.4]{sethuraman2006diffusive} for the nearest-neighbor case. We extend it to the general case.

\begin{proof}
Since
\[\|f\|^2_{1,\lambda,\gens} = \|f\|^2_{1,\lambda,\gens^e} + \<f,(-\gens^t)f\>_{{\nu_\rho^\star}} \quad \text{and} \quad \<f,(-\gens^t)f\>_{{\nu_\rho^\star}} \geq 0,\]
it is obvious that $\|f\|_{1,\lambda,\gens^e} \leq \|f\|_{1,\lambda,\gens}$. To conclude the proof for the $H_1$ bound, we only need to prove
\begin{equation}\label{appen1}
\<f,(-\gens^t)f\>_{{\nu_\rho^\star}} \leq C \<f,(-\gens^e)f\>_{{\nu_\rho^\star}}
\end{equation}
for some constant $C$. Observe that
\begin{align*}
	\<f,(-\gens^t)f\>_{{\nu_\rho^\star}} &= \frac{1}{2} \sum_{z \in \Z_\star^d} s(z) E_{\nu_\rho^\star} \big[(1-\xi(z)) \big(f(\theta_z \xi) - f(\xi)\big)^2\big]\\
	&\leq  \frac{1}{2} \sum_{z \in \Z_\star^d} \sum_{A \subset \Z^d_\star,\atop |A|=n} s(z)   \big[ \mf{f}(\theta_{-z} A) - \mf{f}(A)\big]^2.
\end{align*}
The last inequality comes from the fact that if $\xi (z) = 0$, then $\Psi_A (\theta_z \xi) = \Psi_{\theta_z A} (\xi)$. If $z \notin A$, denote $A = \{x_1,\ldots,x_n\}$ with $x_j \neq z$ for $1 \leq j \leq n$. Then,  $\theta_{-z} A = \{x_1-z,\ldots,x_n-z\}$. We need to move the particles from $A$ to $\theta_{-z} A$.
The idea is as follows: we first move the particle from $x_1$ to $x_1-z$, then from $x_2$ to $x_2 -z$, $\ldots$, and last from $x_n$ to $x_n-z$. Precisely speaking,  let
\[A_0 = A, \quad A_j = \{x_1-z, x_2 - z, \ldots, x_j-z, x_{j+1}, \ldots, x_n\}, \;1 \leq j \leq n.\]
By Cauchy-Schwarz inequality,
\begin{multline*}
	\sum_{z \in \Z_\star^d} \sum_{A \subset \Z^d_\star,\atop z \notin A, |A|=n} s(z)   \big[ \mf{f}(\theta_{-z} A) - \mf{f}(A)\big]^2 \leq  n	\sum_{z \in \Z_\star^d} \sum_{A \subset \Z^d_\star,\atop z \notin A, |A|=n} \sum_{j=0}^{n-1} s(z)  [\mf{f}(A_{j+1}) - \mf{f}(A_j)]^2\\
	\leq n^2 \sum_{x,y \in \Z_\star^d} \sum_{A \subset \Z^d_\star,\atop |A|=n}s(y-x) [\mf{f}(A_{x,y}) - \mf{f}(A)]^2.
\end{multline*}
Similarly, if $z \in A$, denote $A = \{x_1,\ldots,x_{n-1},z\}$, then $\theta_{-z} A = \{x_1-z,\ldots,x_{n-1}-z,-z\}$, and in the last step we move the particle from site $z$ to site $-z$.   Then,
\begin{multline*}
	\sum_{z \in \Z_\star^d} \sum_{A \subset \Z^d_\star,\atop z \in A, |A|=n} s(z)   \big[ \mf{f}(\theta_{-z} A) - \mf{f}(A)\big]^2 \leq  
	n^2 \sum_{x,y \in \Z_\star^d} \sum_{A \subset \Z^d_\star,\atop |A|=n}s(y-x) [\mf{f}(A_{x,y}) - \mf{f}(A)]^2\\
	+ n \sum_{z \in \Z_\star^d} \sum_{A \subset \Z^d_\star,\atop |A|=n}  s (z) \big[ \mf{f}( A_{z,-z}) - \mf{f}(A)\big]^2.
\end{multline*}
Since $s(z) = 2^{d+\alpha} s(2z)$, the last line is bounded by
\[n 2^{d+\alpha} \sum_{x,y \in \Z_\star^d} \sum_{A \subset \Z^d_\star,\atop |A|=n}s(y-x) [\mf{f}(A_{x,y}) - \mf{f}(A)]^2.\]
By \eqref{dirichlet}, there exists a constant $C = C(d,n,\alpha)$ such that $\<f,(-\gens^t)f\>_{{\nu_\rho^\star}} \leq C \<f,(-\gens^e)f\>_{{\nu_\rho^\star}}$.  This concludes the proof for the bound of  $H_1$ norm. 

The inequality associated to $H_{-1}$ norm follows  from the definition of the $H_{-1}$ norm defined in Subsection \ref{subsec:norms} directly, and we leave it to the readers.
\end{proof}

Next, we extend the underlying space $\Z_\star^d$ to $\Z^d$. We mainly focus on degree one functions.   Let $\bar{\mc{E}}_d$ be the family of finite subsets of $\Z^d$, and let $\bar{\mc{E}}_{d,n}$ be those subsets with cardinality $n$.

For degree one coefficient function $\mf{f}: \mathcal{E}_{d,1} \rightarrow \R$, define the extensions $\mf{f}_{\rm ext}$, $\mf{f}_{\odot}: \bar{\mc{E}}_{d,1} \rightarrow \R$ respectively by
\[\mf{f}_{\rm ext} (\{x\}) = \begin{cases}
	\mf{f} (\{x\}) \quad &\text{if } x \neq 0;\\
	\sum_{x \in \Z^d_\star} s(x) \mf{f} (\{x\}) \quad &\text{if } x = 0,
\end{cases}\]   
and 
\[\mf{f}_{\odot} (\{x\}) = \begin{cases}
	\mf{f} (\{x\}) \quad &\text{if } x \neq 0;\\
	0 \quad &\text{if } x = 0.
\end{cases}\]   
We also extend the operator $\mf{S}^e$ to $\mf{S}_{\rm ext}$, which acts on local coefficient functions $\mf{g}: \bar{\mc{E}}_d \rightarrow \R$, as
\[\mf{S}_{\rm ext} \mf{g} (A) = \frac{1}{2} \sum_{x,y \in \Z^d} s(y-x) [\mf{g} (A_{x,y}) - \mf{g} (A)].\]

The following result is an analogue of \cite[Proposition 3.6]{sethuraman2006diffusive}.

\begin{lemma}\label{lem-a2}
There exists a finite constant $C = C(d,\alpha) $ such that for any $\lambda > 0$ and any local coefficient functions $\mf{f}: \mc{E}_d \rightarrow \R$ with degree one,
\[\|\mf{f}\|_{1,\lambda,\mf{S}^e}  \leq \|\mf{f}_{\rm ext}\|_{1,\lambda,\mf{S}_{\rm ext}} \leq C \|\mf{f}\|_{1,\lambda,\mf{S}^e},\]
and
\[ \|\mf{f}\|_{-1,\lambda,\mf{S}^e}  \leq C \|\mf{f}_{\odot}\|_{-1,\lambda,\mf{S}_{\rm ext}}. \]
\end{lemma}

\begin{proof}
Since $\mf{f}$ has degree one, by the definitions of $H_1$ norm and $\mf{f}_{\rm ext}$,
\begin{multline*}
\|\mf{f}_{\rm ext}\|^2_{1,\lambda,\mf{S}_{\rm ext}} = \lambda \sum_{x \in \Z^d} \mf{f}_{\rm ext}^2 (\{x\})  + \frac{1}{4} \sum_{x,y\in \Z^d} s(y-x) \big[\mf{f}_{\rm ext} (\{y\}) - \mf{f}_{\rm ext} (\{x\}) \big]^2\\
=  \lambda \Big\{ \sum_{x \in \Z^d_\star} \mf{f}^2 (\{x\}) + \Big(\sum_{x \in \Z^d_\star} s(x) \mf{f} (\{x\})\Big)^2\Big\} 
+ \frac{1}{4} \sum_{x,y\in \Z^d_\star} s(y-x) \big[\mf{f} (\{y\}) - \mf{f} (\{x\}) \big]^2 \\+ \frac{1}{2} \sum_{x\in \Z^d_\star} s(x) \big[ \sum_{y \in \Z_\star^d} s(y) \mf{f} (\{y\}) - \mf{f} (\{x\}) \big]^2.
\end{multline*}
Obviously, $\|\mf{f}\|_{1,\lambda,\mf{S}^e}  \leq \|\mf{f}_{\rm ext}\|_{1,\lambda,\mf{S}_{\rm ext}}$. By Cauchy-Schwarz inequality and the fact that $s(x) \leq 1$,
\begin{multline*}
	\|\mf{f}_{\rm ext}\|^2_{1,\lambda,\mf{S}_{\rm ext}} \leq 2 \lambda \sum_{x \in \Z^d_\star} \mf{f}^2 (\{x\})  + \frac{1}{4} \sum_{x,y\in \Z^d_\star} s(y-x) \big[\mf{f} (\{y\}) - \mf{f} (\{x\}) \big]^2 \\
	+ \frac{1}{2} \sum_{x,y\in \Z^d_\star} s(x) s(y)\big[  \mf{f} (\{y\}) - \mf{f} (\{x\}) \big]^2.
\end{multline*}
Since $s(x) s(y) \leq 2^{d+\alpha} s(x-y)$ for $x,y \in \Z_\star^d$ and $x \neq y$,  
\[	\|\mf{f}_{\rm ext}\|^2_{1,\lambda,\mf{S}_{\rm ext}}  \leq 2 \lambda \sum_{x \in \Z^d_\star} \mf{f}^2 (\{x\})  + \frac{1+2^{d+\alpha+1}}{4}  \sum_{x,y\in \Z^d_\star} s(y-x) \big[\mf{f} (\{y\}) - \mf{f} (\{x\}) \big]^2 \leq C \|\mf{f}\|^2_{1,\lambda,\mf{S}^e}\]
for some constant $C=C(d,\alpha)$.  This concludes the proof for the $H_1$ bound. 

For the $H_{-1}$ bound,
\[	\|\mf{f}\|^2_{-1,\lambda,\mf{S}^e} = \sup_{\mf{g}} \Big\{ 2 \<\mf{f},\mf{g}\> -  \|\mf{g}\|_{1,\lambda,\mf{S}^e}^2\Big\}
\leq \sup_{\mf{g}} \Big\{ 2 \<\mf{f}_{\odot},\mf{g}_{\rm ext}\> -  C^{-2} \|\mf{g}_{\rm ext}\|_{1,\lambda,\mf{S}_{\rm ext}}^2\Big\} \leq C^2 \|\mf{f}_{\odot}\|^2_{-1,\lambda,\mf{S}_{\rm ext}},\]
where the above supremum is over coefficient functions $\mf{g}: \mc{E}_d \rightarrow \R$ which are $L^2$ integrable with degree one. This completes the proof.
\end{proof}

\bibliographystyle{plain}
\bibliography{zhaoreference.bib}
\end{document}